\documentclass[dvipsnames,a4paper,twoside]{article}
\usepackage[bookmarks=true,bookmarksopen=true]{hyperref}
\usepackage{array,float}

\newcommand{\norm}[1]{\left\|#1\right\|}
\newcommand{\innOm}[2]{(#1,#2)_\Omega}
\newcommand{\innGa}[2]{\langle #1,#2\rangle_\Gamma}

\newcommand{\duaGa}[2]{[#1,#2]_\Gamma}
\usepackage{amssymb,amsfonts,amsthm,mathrsfs}
\usepackage{amsmath,bm,xcolor}
\usepackage{multirow}
\usepackage{makecell}

\theoremstyle{plain}% Theorem-like structures provided by amsthm.sty
\newtheorem{theorem}{Theorem}[section]
\newtheorem{lemma}{Lemma}[section]
\newtheorem{corollary}{Corollary}[section]
\newtheorem{proposition}[theorem]{Proposition}

\theoremstyle{definition}

\newtheorem{example}[theorem]{Example}

\theoremstyle{remark}
\newtheorem{remark}[theorem]{Remark}

\numberwithin{equation}{section}

	\title{A New Global Divergence Free and  Pressure-Robust  HDG  Method for  Tangential   Boundary Control of Stokes Equations\thanks{G. Chen is supported by National Natural Science Foundation of China (NSFC) under grant no. 11801063  and no. 121713413,
				the Fundamental Research Funds for the Central Universities grant no. YJ202030. W. Gong is supported by the Strategic Priority Research Program of Chinese Academy of Sciences (Grant No. XDB 41000000), the National Key Basic Research Program (Grant No. 2018YFB0704304) and the National Natural Science Foundation of China (Grant No. 12071468 and 11671391). M. Mateos is supported by MCIN/ AEI/10.13039/501100011033/ under research project PID2020-114837GB-I00. J.\ Singler is supported by the US National Science Foundation (NSF) under grant number 2111421. Y.\ Zhang is supported by the US National Science Foundation (NSF) under grant number 2111315.}}

\author{%
{\sc Gang Chen}\thanks{ School of Mathematics, Sichuan University, Chengdu, China {\tt cglwdm@scu.edu.cn}}\and
{\sc Wei Gong}\thanks{ NCMIS \& LSEC, Institute of Computational Mathematics and Scientific/Engineering Computing, Academy of Mathematics and Systems Science, Chinese Academy of Sciences, Beijing, China {\tt wgong@lsec.cc.ac.cn}}\and
{\sc Mariano Mateos}\thanks{ Dpto.\ de Matem\'aticas. Universidad de Oviedo,
			Campus de Gij\'on, Spain {\tt mmateos@uniovi.es}}\and
{\sc John R.\ Singler}\thanks{ Department of Mathematics and Statistics, Missouri University of Science and Technology,
			Rolla, MO, USA {\tt singlerj@mst.edu}}\and
{\sc Yangwen Zhang}\thanks{ Department of Mathematical Science, Carnegie Mellon University,
			Pittsburgh, PA, USA {\tt yangwenz@andrew.cmu.edu}}
}

\begin{document}

\maketitle
	\begin{abstract}
		In {[ESAIM: M2AN, 54(2020), 2229-2264]}, we proposed an HDG method to approximate the solution of a tangential boundary control problem for the Stokes equations and obtained an optimal convergence rate for the optimal control {that reflects its global regularity}. However, the error estimates depend on the pressure, and the velocity is not divergence free. The importance of pressure-robust numerical methods for fluids was addressed by John et al.\   {[SIAM Review, 59(2017), 492-544]}. In this work, we devise a new HDG method to approximate the solution of the Stokes tangential boundary control problem; the HDG method is also of independent interest for solving the Stokes equations. This scheme yields a $\bm H(\textup{div})$ conforming, globally divergence free, and pressure-robust {solution}.  To the best of our knowledge, this is the first time such a numerical scheme has been obtained for an optimal boundary control problem for the Stokes equations. We also provide numerical experiments to show the performance of the new HDG method and the advantage over the non pressure-robust scheme.
	\end{abstract}
	
\begin{quote}\textbf{Keywords} Dirichlet optimal control; Stokes system; Hybridizable Discontinous Galerkin method; pressure-robust method
\end{quote}	

\begin{quote}
\end{quote}

	\section{Introduction}
	Control of fluid flows modeled by Stokes or Navier-Stokes equations is an important area of research that has undergone major developments %both theoretically and computationally
in the recent past.  The model poses many theoretical and computational challenges and there is an extensive body of literature devoted to this subject; see, e.g., \cite{MR1613897,MR1871460,MR1885709,MR2051068,MR2326293,MR2974748,MR3233093,MR3357635,MR3663000}.
	In \cite{GongStokes_Tangential1} we investigated an HDG discretization for the tangential boundary control of a fluid governed by the Stokes system and {proved optimal error estimates with respect to the global regularity of the optimal control}; however, the numerical method is not pressure-robust, i.e., the discretization {errors} depend on the norm of the pressure.

  As pointed {out} by John et al.\ in the 2017 review {article}  \cite{MR3683678}, many mixed finite element methods, such as  Taylor-Hood finite element, Crouzeix-Raviart and  MINI elements are not pressure-robust.
The key for a numerical scheme to be pressure-robust is the way the null divergence condition is discretized. In the {above} mentioned review, at least three ways to obtain pressure-robust mixed methods are described: building $H^1$-conforming divergence-free schemes, using discontinuous Galerkin methods, or committing some variational crime.  In 2014, Linke \cite{MR3133522} slightly {modified} the classical lowest order Crouzeix-Raviart element with a variational crime by noticing that the Raviart-Thomas interpolation -- see \eqref{RT} below -- maps divergence-free vector fields onto divergence-free discrete vector fields. In this way, the discrete velocity of the numerical solution is not affected when the external force is modified with a gradient field, which is a property that is satisfied by the continuous solution: if $-\Delta\bm y +\nabla p = \bm f$ and $\nabla\cdot\bm y = 0$, then for any scalar field $\phi$, $-\Delta\bm y+\nabla(p+\phi) = \bm f+\nabla\phi$, $\nabla\cdot \bm y=0$ and only the pressure is modified. In 2007, Cockburn et al.~\cite{MR2304270} had already studied a DG method for the Navier-Stokes equations which yields divergence-free solutions.

%In 2004, Dawson et al. \cite{MR2055253} investigated numerical schemes for flow problems coupled with the transport equation. They showed that, for the methods used for the transport equation to satisfy certain desirable properties --namely, being zeroth-order accurate and globally conservative--, the method used for the coupled flow must satisfy a compatibility condition, which in some cases means that the discrete flow must be globally conservative.

Hybridizable discontinuous Galerkin (HDG) methods were proposed by Cockburn et al. in \cite{MR2485455} as an improvement of traditional DG methods; for a recent didactic exposition, see, e.g., \cite{Cockburn2021}.
  The HDG algorithm proposed and analyzed in our work \cite{GongStokes_Tangential1} is not pressure-robust:  although the convergence rate is optimal, the magnitude of the error strongly depends on the pressures; see Example \ref{example1} below. %In this work, we devise a numerical method that preserves the zero-divergence condition of the velocity at the discrete level to guarantee that the velocity and control errors are not dependent on the pressures.

	In 2016, Lehrenfeld and Sch\"{o}berl \cite{MR3511719} first proposed a pressure-robust HDG method for  the Navier-Stokes equations and used  a divergence-conforming velocity space; \textcolor{black}{see also Lederer, Lehrenfeld, and Sch\"{o}berl \cite{MR3942180} for an improvement of this method.}  Recently,
%Kirk and Rhebergen in \cite{MR4019980}
\textcolor{black}{Rhebergen and Wells, in \cite{MR3833698}}, used standard cell and facet discontinuous Galerkin spaces {that do} not involve a divergence-conforming finite element space for the velocity. They obtained pressure-robust scheme for {the} Navier-Stokes equation; \textcolor{black}{see also Kirk and Rhebergen in \cite{MR4019980} for a detailed analysis of this method}.  For other pressure-robust HDG methods, see \cite{baierreinio2021analysis,RhebergenWells2020,rhebergen2021preconditioning,MR4161756}.
In this paper, we propose a new HDG scheme with less degrees of freedom than that of \cite{MR3511719}, apply it to a tangential boundary control problem governed by the Stokes equation, and prove that the method is pressure-robust.

% \textcolor{red}{Despite the large amount of existing work on numerical methods for fluid flow control problems, the authors are not aware of any existing work that ensures a globally divergence free velocity. As a first attempt, we propose a pressure-robust scheme for solving the following tangential boundary control problem:}

Despite the large amount of existing work on numerical methods for fluid flow control problems, the authors are only aware of one work dealing with pressure-robustness in the context of optimal control problems, the very recent preprint \cite{Merdon-Wollner2022}, where a distributed control problem governed by the Stokes equation is discretized by means of a pressure-robust variant of a classical finite element discretization. We, on the other hand,  propose a pressure-robust HDG scheme for solving the following tangential boundary control problem:
	\begin{align}\label{Ori_problem1}
		\min\limits_{\bm u\in \bm U} J(\bm u)=\frac{1}{2}\|\bm y_{\bm u}-\bm y_{d}\|^2_{\bm L^{2}(\Omega)}+\frac{\gamma}{2}\| \bm u\|^2_{\bm U},
	\end{align}
	where $\bm y_{\bm d}$ is the desired state, $\bm y_{\bm u}$ is the unique solution in the transposition sense (see, e.g., \cite[Defintion 2.3]{GongStokes_Tangential1}) of
	\begin{align}\label{Ori_problem2}
		-\Delta\bm y+\nabla p =\bm f \;  \text{in}\ \Omega, \;\;\;\;
		\nabla\cdot\bm y=0 \;\  \text{in}\ \Omega, \;\;\;\;
		\bm y=\bm u \  \text{on}\  \Gamma, \;\;\;\;
		\int_{\Omega} p=0,
	\end{align}
    $\gamma$ is a positive constant, and we take the control space
	\[\bm U = \{\bm u = u\bm\tau:\, u\in L^2(\Gamma)\}\]
	with norm $\|\bm u\|_{\bm U} = \|u\|_{L^2(\Gamma)}$ and  $\bm \tau$ the unit tangential vector.
	
	Formally, the optimal control $u\in L^2(\Gamma)$ and the optimal state $\bm y \in  \bm L^2(\Omega)$ satisfy the {first order} optimality system
	\begin{subequations}\label{optimality_system}
		\begin{gather}
			-\Delta \bm y + \nabla p = \bm f \; \textup{in} \ \Omega,\;\;\;\;\nabla\cdot\bm y = 0\; \textup{in} \ \Omega,\;\;\;\;\bm y = u\bm \tau \; \textup{on} \ \Gamma,\\
			-\Delta \bm z -\nabla q = \bm y - \bm y_d \; \textup{in} \ \Omega,\;\;\;\;
			\nabla\cdot\bm z = 0\;  \textup{in} \ \Omega,\;\;\;\;
			\bm z = 0   \; \textup{on} \ \Gamma,\\
			%
			%
			%
%	\textcolor{red}{		\partial_{\bm n}  \bm z = \gamma  u\bm\tau \; \textup{on} \  \Gamma.}\nonumber\\
		\textcolor{black}{	\partial_{\bm n}  \bm z\cdot \bm\tau = \gamma  u \; \textup{on} \  \Gamma.}
		\end{gather}
	\end{subequations}	
	In \cite{GongStokes_Tangential1},  we proved that  the optimal control is indeed determined by a {very weak} formulation of the above optimality system and we proved a regularity result for the solution in 2D {polygonal domains}.  The optimal control satisfies (see \cite[Theorem 2.4]{GongStokes_Tangential1}) $u\in H^s(\Gamma)$ with $s\in(0,3/2)$. We utilized an existing HDG method to discretize the optimality system and obtained the following a priori error estimate (see \cite[Theorem 4.1]{GongStokes_Tangential1}):
	\begin{align}\label{error_u}
		\| u -  u_h\|_{L^2(\Gamma)}  \le {C h^{s}(\|\bm y\|_{\bm H^{s+1/2}(\Omega)}+\|\bm z\|_{\bm H^{s+3/2}(\Omega)}+\|p\|_{H^{s-1/2}(\Omega)}+\|q\|_{H^{s+1/2}(\Omega)}+\|u\|_{H^s(\Gamma)})}.
	\end{align}
	The error estimate \eqref{error_u} implies that the error is dependent on the pressure $p$ and dual pressure $q$.

	In this paper, we propose a new HDG method to revisit  the problem \eqref{Ori_problem1}-\eqref{Ori_problem2}.  Our new HDG method is pressure-robust; i.e., we obtain the a priori error estimate (see Theorem \ref{main_res}):
	\begin{align}\label{error_u2}
	\| u -  u_h\|_{L^2(\Gamma)}  \le C h^{s}(\|\bm y\|_{\bm H^{s+1/2}(\Omega)}+\|\bm z\|_{\bm H^{s+3/2}(\Omega)}).
	\end{align}
	The error estimate \eqref{error_u2} shows the same convergence rate as obtained in \cite{GongStokes_Tangential1}, but the errors no longer depend on the pressures.

\textcolor{black}{As in \cite{MR3833698}, our} method  introduces a numerical trace to approximate the pressure on the boundary edge, \textcolor{black}{but in that reference,} the authors use polynomials of degree $k+1$ to approximate the trace of the velocity and we use polynomials of degree $k$. Hence, the degrees of freedoms of our scheme are  less than that in \cite{MR3833698}. \textcolor{black}{The price, of course, is that we obtain lower orders of convergence than those obtained in \cite{MR4019980} for the method proposed in \cite{MR3833698}, but on the other hand, our error estimates are valid for problems with very low regularity solutions, as the ones we find when solving Dirichlet control problems}.

\textcolor{black}{We find that a pressure-robust method is specially appropriate for the tangential control problem that we address. Notice that if we perturb $\bm y_d$ with a conservative field $\nabla\phi$ for some scalar function $\phi$, the optimal solution would not change at all. We sould just replace $q$ by $q+\phi$ to obtain the solution of the optimality system.}

    The plan of this paper is as follows. In Section \ref{sec:Regularity and HDG Formulation} we present the functional framework, the optimality system for the control problem, and the new HDG formulation; we prove that, for any given control, both the discrete velocity and adjoint velocity are divergence free.  Section \ref{sec:analysis} is devoted to the error analysis; we present and prove our main result. The scheme of our proof largely follows the structure in our previous work \cite{GongStokes_Tangential1}, but here we needed to use new techniques to show in every auxiliary lemma that the obtained estimates are independent of the pressure. Finally, in  Section \ref{Numerical_experiments} we provide the results of two numerical experiments to compare the performance of the present pressure-robust method with the method in \cite{GongStokes_Tangential1}.

	\section{Background:  Regularity  and HDG Formulation }
	\label{sec:Regularity and HDG Formulation}
	In this section, we briefly review the regularity results for the tangential boundary control problem and give the HDG formulation.
	
	First, we define some  notation{. Let $\Omega$ be a bounded  polygonal domain. We} use the standard notation  $H^{m}(\Omega)$ {to} denote the Sobolev space with norm $\|\cdot\|_{m,\Omega}${. In} many places, we  use  $\|\cdot\|_{m}$ to replace $\|\cdot\|_{m,\Omega}$ if the context makes the norm clear.   Let  $\mathbb H^m(\Omega) = [H^{m}(\Omega)]^{2\times 2}$, $\bm H^{m}(\Omega) = [H^{m}(\Omega)]^2$ and $\bm H_0^1(\Omega) =\{\bm v\in \bm H^1(\Omega); \bm v = 0 \ \textup{on} \ \Gamma \}$. Let $\innGa{\cdot}{\cdot}$ denote the inner product in $L^2(\Gamma)$ and let $\duaGa{\cdot}{\cdot}$ denote the duality product between $H^{-s}(\Gamma)$ and $H^s(\Gamma)${.  We} introduce  the spaces
	\begin{align*}
		\bm V^s(\Omega) &= \{\bm y\in \bm H^s(\Omega) : \nabla \cdot \bm y = 0, \  \duaGa{\bm y\cdot\bm n}{1}=0\},\mbox{ for }s\geq 0,\\
		\bm V^s_0(\Omega) &= \{\bm y\in \bm H^s(\Omega) : \nabla \cdot  \bm y = 0, \  \bm y=0\mbox{ on }\Gamma\},\mbox{ for }s>1/2,\\
		\bm V^s(\Gamma) &=  \{\bm u\in \bm H^s(\Gamma) : \innGa{\bm u\cdot\bm n}{1} = 0\},\mbox{ for }0\leq s<3/2.
	\end{align*}
	
	We denote the $L^2$-inner products on $ \mathbb L^2(\Omega) $,   $\bm L^2(\Omega)$, $ L^2(\Omega)$ and $ \bm L^2(\Gamma)$ by
	\begin{align*}
		(\mathbb L,\mathbb G)_{\Omega} = \sum_{i,j=1}^2 \int_{\Omega}  L_{ij} G_{ij}, \;\;  (\bm y,\bm z)_{\Omega} = \sum_{j=1}^2 \int_{\Omega}  y_j z_j,\;\;
		(p,q)_{\Omega} =  \int_{\Omega} pq, \;\;  \langle \bm y,\bm z\rangle_{\Gamma} = \sum_{j=1}^2\int_{\Gamma}  y_j  z_j.
	\end{align*}
	Define the spaces $\mathbb H(\text{div};\Omega)$ and $ L_0^2(\Omega)$ as
	\begin{align*}
		\mathbb H(\text{div},\Omega) = \{\mathbb K \in \mathbb L^2(\Omega), \nabla\cdot \mathbb K \in \bm L^2(\Omega)\}, \quad
		L_0^2(\Omega) = \left\{p \in L^2(\Omega),  (p,1)_{\Omega} = 0\right\}.
	\end{align*}

	\subsection{Regularity}
	\label{sec:regularity}
	In \cite[Theorem 2.8 and Corollary 2.9]{GongStokes_Tangential1}, we proved the following well-posedness and regularity result for the tangential Dirichlet boundary control problem \eqref{Ori_problem1} - \eqref{Ori_problem2}. Set $\mathbb L = \nabla \bm y$ and $\mathbb G = \nabla \bm z$, let $\omega$ be the largest interior angle of $\Gamma$, and let $\xi \in(0.5,4]$ {be} the real part of the smallest root different from zero  of the equation
	\begin{align}\label{singular_ex}
	\sin^2(\lambda\omega)-\lambda^2\sin^2\omega=0.
	\end{align}
	It is known that $ \xi > \pi / \omega $ \textcolor{black}{if} $ \omega < \pi $ and $ 0.5 < \xi < \pi/\omega $ if $ \omega > \pi $.
	
	% Now we can state the regularity for the solution of problem \eqref{Ori_problem1}-\eqref{Ori_problem2}.
	\begin{theorem}\label{regularity_res}
		If $\Omega$ is a convex polygonal domain, $\bm f \in\bm L^2(\Omega)$ and $\bm y_d\in \bm H^{\min\{2,\xi\} }(\Omega)$, then there is a unique solution $u\in L^2(\Gamma)$ of problem \eqref{Ori_problem1}-\eqref{Ori_problem2}. The solution $ u $ satisfies
		$
		u\in H^s(\Gamma)
		%\mbox{$u\in H^s(\Gamma)$ for all $1/2< s<\min\{3/2,\xi-1/2\}$}
		$
		for all $1/2< s<\min\{3/2,\xi-1/2\}$ and there exists
		\begin{align*}
			\bm y &\in \bm V^{s+1/2}(\Omega),  &  \mathbb{L} &\in \mathbb  H^{s-1/2}(\Omega),  &  p &\in H^{s-1/2}(\Omega)\cap L_0^2(\Omega),\\
			\bm z &\in \bm V_0^{r+1}(\Omega),  &  \mathbb{G} &\in \mathbb H^r(\Omega),  &  q &\in H^{r}(\Omega)\cap L_0^2(\Omega)
		\end{align*}
		for all $1 <r<\min\{3,\xi\}$, and $\mathbb L - p\mathbb I \in \mathbb H(\textup{div}, \Omega)$ such that
		\begin{subequations}
			\begin{align}
				\innOm{\mathbb{L}}{\mathbb{T}}+\innOm{\bm y}{\nabla \cdot\mathbb{T}} &= \innGa{u\bm\tau}{\mathbb{T}\bm n}, \label{TCOSE1}\\
				-\innOm{\nabla\cdot(\mathbb{L}- p\mathbb I)}{ \bm v} &=\innOm{\bm f}{\bm v},\label{TCOSE2} \\
				(\nabla\cdot\bm y, w)_{\Omega} &=0,\label{TCOSE3}\\
				\innOm{\mathbb{G}}{\mathbb{T}} +\innOm{\bm z}{\nabla\cdot\mathbb{T}}&=0,\label{TCOASE1}\\
				-\innOm{\nabla\cdot(\mathbb{G} + q\mathbb I)}{\bm v} &= \innOm{\bm y-\bm y_d}{\bm v}, \label{TCOASE2}\\
				\innOm{\nabla \cdot \bm z}{ w} &=0, \label{TCOASE3}\\
				\innGa{\gamma  u\bm\tau-\mathbb{G} \bm n}{\mu\bm{\tau}} &= 0 \label{TCOOC}
			\end{align}
			for all $(\mathbb T, \bm v, w, \mu )\in \mathbb H(\textup{div}, \Omega)\times \bm L^2(\Omega)\times L_0^2(\Omega)\times L^2(\Gamma)$.
		\end{subequations}
		Moreover,
		\begin{equation}\label{eqn:opt_control_local_reg}
			u\in \prod_{i=1}^{m} H^{r-1/2}(\Gamma_i)\mbox{ for all }r<\min\{3,\xi\},
		\end{equation}
		{where $\Gamma_i$ denotes the smooth segment of $\Gamma$ such that $\Gamma=\bigcup\limits_{i=1}^m\Gamma_i$.}
	\end{theorem}

	\subsection{The HDG Formulation}
	%\subsection{Mesh and Approximation spaces}
	\label{sec:HDG}
	
	We use the same notation as in \cite{GongStokes_Tangential1} to describe the HDG method.   Let $ \{ \mathcal T_h \} $ be a {family of} conforming {and} quasi-uniform triangular {meshes} of $ \Omega $. This assumption on the meshes is stronger than in \cite{GongStokes_Tangential1}; there we assumed $ \{ \mathcal T_h \} $ is a {family of} conforming {and} quasi-uniform polygonal {meshes}. Let $\partial \mathcal{T}_h$ denote the set $\{\partial K: K\in \mathcal{T}_h\}$. For an element $K$ of the collection  $\mathcal{T}_h$, $e = \partial K \cap \Gamma$ is the boundary edge if the length of $e$ is non-zero. For two elements $K^+$ and $K^-$ of the collection $\mathcal{T}_h$, $e = \partial K^+ \cap \partial K^-$ is the interior edge between $K^+$ and $K^-$ if the length of $e$ is non-zero. Let $\mathcal E_h^o$ and $\mathcal E_h^{\partial}$ denote the set of interior and boundary edges, respectively. We denote by $\mathcal E_h$ the union of  $\mathcal E_h^o$ and $\mathcal E_h^{\partial}$. We  introduce various inner products:
	\begin{gather*}
		(\eta,\zeta)_{\mathcal{T}_h} = \sum_{K\in\mathcal{T}_h} (\eta,\zeta)_K, \quad
		(\bm\eta,\bm\zeta)_{\mathcal{T}_h} = \sum_{i=1}^2 (\eta_i,\zeta_i)_{\mathcal T_h}, \quad	(\mathbb L,\mathbb G)_{\mathcal{T}_h} = \sum_{i,j=1}^2 (L_{ij},G_{ij})_{\mathcal T_h},\\
		\langle \eta,\zeta\rangle_{\partial\mathcal{T}_h} = \sum_{K\in\mathcal T_h} \langle \eta,\zeta\rangle_{\partial K}, \quad \langle \bm \eta,\bm\zeta\rangle_{\partial\mathcal{T}_h} = \sum_{i=1}^2 \langle \eta_i,\zeta_i\rangle_{\partial \mathcal T_h}.
	\end{gather*}
	{The norms induced by the above inner products are defined accordingly.}
	
	Let $\mathcal{P}^k(D)$ denote the set of polynomials of degree at most $k$ on a domain $D$.  We introduce the following discontinuous finite element spaces:
	\begin{align*}
		\mathbb K_h&:=\{\mathbb L \in \mathbb L^2(\Omega):\mathbb L|_K\in[\mathcal P^k(K)]^{2\times 2}, \ \forall K\in\mathcal T_h\},\\
		\bm{V}_h  &:= \{\bm{v}\in \bm L^2(\Omega): \bm{v}|_{K}\in [\mathcal{P}^{k+1}(K)]^2, \   \forall K\in \mathcal{T}_h\},\\
		{W}_h  &:= \{{w}\in L^2(\Omega): {w}|_{K}\in \mathcal{P}^k(K),  \ \forall K\in \mathcal{T}_h\},\\
		\bm {M}_h  &:= \{{\mu}\in \bm L^2(\mathcal E_h): {\bm\mu}|_{e}\in [\mathcal{P}^k(e)]^2, \  \forall e\in \mathcal E_h\},\\
		{M}_h  &:= \{{\mu}\in L^2(\mathcal E_h^\partial): {\mu}|_{e}\in \mathcal{P}^k(e), \  \forall e\in \mathcal E_h^\partial\},\\
		{Q}_h  &:={ \{{\mu}\in L^2(\mathcal E_h): {\mu}|_{e}\in \mathcal{P}^{k+1}(e), \  \forall e\in \mathcal E_h\}}.
	\end{align*}

	Let  $\bm M_h(o)$ denote the space defined in the same way as $\bm M_h$, but with $\mathcal E_h$ replaced by $\mathcal E_h^o$. We use $\nabla \bm v$ and $ \nabla \cdot \mathbb L $ to denote the gradient of $ \bm v $ and the divergence of $ \mathbb L $ taken piecewise on each element $K\in \mathcal T_h$. Finally, we define
	\begin{align*}
		{W}_h^0  &= \left\{{w}\in L^2(\Omega): {w}|_{K}\in \mathcal{P}^k(K), \  \forall K\in \mathcal{T}_h \ \textup{and}\  (w,1)_{\Omega}  = 0\right\}.
	\end{align*}
	
	%To approximate the solution of the mixed weak form \eqref{TCOSE1}-\eqref{TCOOC}, the HDG formulation considered here is modified from Part I \cite{GongStokes_Tangential1} to avoid the estimation of $\mathbb L \in \mathbb H^{r_{\mathbb L}}(\Omega)$ and $p\in H^{r_p}(\Omega)$ on the boundary when $r_{\mathbb L}<1/2$ and $r_p<1/2$.
	
	The HDG method seeks approximate fluxes $\mathbb L_h,\mathbb G_h \in \mathbb {K}_h $, states $ \bm y_h, \bm z_h \in \bm V_h $, pressures  $p_h,q_h \in W_h^0$,  interior element boundary traces $ \widehat{\bm y}_h^o,\widehat{\bm z}_h^o \in \bm M_h(o) $ {and $\widehat{p}_h,\widehat{q}_h\in Q_h$}, and boundary control $ u_h\in  M_h$ satisfying
	\begin{subequations}\label{HDG_discrete2}
		\begin{align}
			(\mathbb L_h,\mathbb T_1)_{\mathcal{T}_h}+(\bm y_h,\nabla\cdot\mathbb T_1)_{\mathcal{T}_h}-\left\langle \widehat{\bm y}_h^o, \mathbb T_1 \bm{n}\right\rangle_{\partial\mathcal{T}_h \backslash\mathcal E_h^\partial} &=\langle u_h\bm \tau,\mathbb T_1 \bm n\rangle_{\mathcal E_h^\partial} \label{HDG_discrete2_a},\\
			-(\nabla\cdot\mathbb L_h,\bm v_1)_{\mathcal T_h}-( p_h,\nabla\cdot \bm v_1)_{\mathcal{T}_h}+\langle \widehat{p}_h,\bm v_1\cdot\bm n \rangle_{\partial\mathcal{T}_h} \nonumber\\
			+\langle h^{-1} P_{\bm M} \bm y_h, \bm v_1 \rangle_{\partial {\mathcal{T}_h}}-\langle  h^{-1} \widehat{\bm y}_h^o, \bm v_1\rangle_{\partial \mathcal{T}_h\backslash\mathcal E_h^\partial} &= (\bm f,\bm v_1)_{\mathcal{T}_h}+\langle h^{-1} u_h\bm \tau ,\bm v_1\rangle_{\mathcal E_h^\partial} ,  \label{HDG_discrete2_b}\\
			(\nabla\cdot\bm y_h, w_1)_{\mathcal T_h}&=0,\label{HDG_discrete2_c}\\
			\langle \bm y_h \cdot\bm n,\widehat{w}_1\rangle_{\partial\mathcal T_h }&=0\label{HDG_discrete2_d}
		\end{align}
		for all $(\mathbb{T}_1,\bm v_1,w_1,\widehat{w}_1)\in \mathbb K_h\times\bm{V}_h\times W_h^0\times Q_h$,
		\begin{align}
			(\mathbb G_h,\mathbb T_2)_{\mathcal{T}_h}+(\bm z_h,\nabla\cdot\mathbb T_2)_{\mathcal{T}_h}-\left\langle \widehat{\bm z}_h^o, \mathbb T_2 \bm{n}\right\rangle_{\partial\mathcal{T}_h\backslash \mathcal E_h^\partial} &=0, \label{HDG_discrete2_e}\\
			-(\nabla\cdot\mathbb G_h,\bm v_2)_{\mathcal T_h}+( q_h,  \nabla\cdot\bm{v}_2)_{\mathcal{T}_h}-\langle \widehat{q}_h,\bm v_2\cdot\bm n \rangle_{\partial\mathcal{T}_h}\nonumber\\
			+\langle h^{-1} P_{\bm M} \bm z_h, \bm v_2\rangle_{\partial {\mathcal{T}_h}}-\langle h^{-1} \widehat{\bm z}_h^o,\bm v_2\rangle_{\partial \mathcal{T}_h\backslash\mathcal E_h^\partial} &= (\bm y_h-\bm y_d,\bm{v}_2)_{\mathcal{T}_h},  \label{HDG_discrete2_f}\\
			(\nabla\cdot\bm z_h, w_2)_{\mathcal T_h}&=0,\label{HDG_discrete2_g}\\
			\langle \bm z_h \cdot\bm n,\widehat{w}_2\rangle_{\partial\mathcal T_h }&=0\label{HDG_discrete2_h}
		\end{align}
		for all $(\mathbb T_2,\bm v_2,w_2,\widehat w_2 )\in \mathbb K_h\times\bm{V}_h\times W_h^0\times Q_h$,
		\begin{align}
			{\langle {\mathbb L}_h \bm n - h^{-1} (\bm P_M\bm y_h - \widehat{\bm y}_h^o),\bm\mu_1\rangle_{\partial\mathcal T_h\backslash\mathcal E_h^\partial}}&=0\label{HDG_discrete2_i}
		\end{align}
		for all $\bm\mu_1\in \bm M_h(o)$,
		\begin{align}
			{\langle {\mathbb G}_h\bm n - h^{-1} (\bm P_M\bm z_h - \widehat{\bm z}_h^o),\bm\mu_2\rangle_{\partial\mathcal T_h\backslash\mathcal E_h^\partial}}&=0\label{HDG_discrete2_j}
		\end{align}
		for all $\bm\mu_2\in \bm M_h(o)$,
		\begin{align}
			{\langle {\mathbb G}_h  \bm n  - h^{-1} \bm P_M\bm z_h- \gamma  u_h\bm{\tau},  \mu_3\bm{\tau} \rangle_{\mathcal E_h^\partial} }&= 0\label{HDG_discrete2_k}
		\end{align}
		for all $\mu_3 \in  M_h $.  Here $\bm P_M$ denotes the standard $L^2$-orthogonal projection from $\bm L^2(\textcolor{black}{\mathcal E_h})$ onto $\bm M_h$; see \eqref{def_P_M} below. This completes the formulation of the HDG method.
	\end{subequations}
	
	\begin{remark}\label{Remark2.2}
{\color{black}Our method resembles the one introduced in \cite{MR3833698} and analyzed in \cite{MR4019980} in the sense that the numerical trace of the pressure plays the role of Lagrange multipliers enforcing continuity of the normal component of the velocity across element boundaries. Nevertheless, to approximate the trace of the velocity, we use polynomials of degree $k$ instead of $k+1$. In this way, our method has fewer degrees of freedom, but at the price of a lower order of convergence. This feature can be seen as a drawback when solving an uncontrolled Stokes problem or even a distributed control problem governed by the Stokes equation. But for the problem at hand the regularity of the solution is usually very low, see Theorem \ref{regularity_res}, and the order of convergence will be mainly limited by this fact, so it makes sense to use a method with suboptimal rates of convergence.}

%Recently,  {Kirk  and Rheberger\ in \cite{MR4019980}}  obtained a pressure-robust HDG method for the Navier-Stokes {equations}.  They {use} polynomials of degree $k+1$ to approximate the trace of the velocity and we use polynomials of degree $k$. Hence, the degrees of freedoms {of our scheme} are less than {that of \cite{MR4019980}}.

		%However,
\textcolor{black}{Notice also that }the HDG {method} developed in this paper
%is computationally more expensive
\textcolor{black}{has more degrees of freedom} than the scheme in \cite{GongStokes_Tangential1}, since we introduced two more numerical traces $\widehat p_h$ and $\widehat q_h$ to approximate the traces of the pressures $p_h$ and $q_h$, respectively \textcolor{black}{in order to obtain a pressure robust method}.
	\end{remark}
	
	Next, we show that the discrete system \eqref{HDG_discrete2} yields a \emph{globally} divergence free state $\bm y_h$ and dual state $\bm z_h$.
	
	\begin{proposition}  \label{prof_divergence_free}
		Let $\bm{y}_h$ and $\bm{z}_h$ be the solutions of \eqref{HDG_discrete2}, then we have $\bm{y}_h, \bm{z}_h\in\bm{H}(\textup{div};\Omega)$ and   $\nabla\cdot\bm y_h=\nabla\cdot\bm{z}_h = 0$.
	\end{proposition}
	
	\begin{proof}
		We only prove the result for $\bm y_h$ since the proof for $\bm z_h$ is similar.  Let $K_1$, $K_2\in \mathcal{T}_h$ be any two adjacent elements sharing a common edge  $e$. Define $\widehat{r}\in Q_h$ as follows:
		\begin{align*}
			\widehat{r}|_e &= -(\bm y_h\cdot \bm n_e)|_{K_1\cap e}-(\bm y_h\cdot \bm n_e)|_{K_2\cap e}\qquad \forall  e\in \mathcal{E}_h^o,\\
			\widehat{r}|_e&=0\qquad\qquad\qquad\qquad\qquad\qquad\qquad\qquad\ \  \forall e\in \mathcal{E}_h^\partial.
		\end{align*}
		Let $c_0=\frac{1}{|\Omega|}\sum\limits_{K\in\mathcal{T}_h}\int_{K} \nabla\cdot \bm y_h$ and  take $(w_1,\widehat{w}_1)=(\nabla\cdot\bm y_h-c_0,\widehat{r}-c_0)$ in {\eqref{HDG_discrete2_c}-\eqref{HDG_discrete2_d}} to get
		\begin{align*}
			0&= -(\nabla\cdot\bm y_h, \nabla\cdot\bm y_h-c_0 )_{\mathcal T_h}+\langle \bm y_h \cdot\bm n,\widehat{r}-c_0\rangle_{\partial\mathcal T_h }\\
			&=-(\nabla\cdot\bm y_h, \nabla\cdot\bm y_h )_{\mathcal T_h}+\langle \bm y_h \cdot\bm n,\widehat{r}\rangle_{\partial\mathcal T_h }\\
			&=-(\nabla\cdot\bm y_h, \nabla\cdot\bm y_h )_{\mathcal T_h}-\sum_{e\in \mathcal{E}_h^o}\| (\bm y_h\cdot\bm n_e)|_{K_1}+(\bm y_h\cdot\bm n_e)|_{K_2} \|_{0,e}^2.
		\end{align*}
		This  implies $\bm{y}_h\in\bm{H}(\textup{div};\Omega)$ and $\nabla\cdot\bm y_h=0$.
	\end{proof}

	\section{Error Analysis}
	\label{sec:analysis}
	We assume {that} the solution of  \eqref{TCOSE1}-\eqref{TCOOC}  satisfies
	\begin{align*}
		\mathbb L \in \mathbb H^{r_{\mathbb L}}(\Omega), \quad   \bm y \in \bm H^{r_{\bm y}}(\Omega), \quad  \mathbb G \in \mathbb H^{r_{\mathbb G}}(\Omega), \quad  \bm z \in \bm H^{r_{\bm z}}(\Omega),
	\end{align*}
	where
	\begin{align}\label{reg_assumption}
		r_{\bm y} > 1, \quad  \ r_{\bm z} > 2,  \quad   \ r_{\mathbb L}> 1/2,  \quad  r_{\mathbb G} > 1.
	\end{align}

	We now state our main result.
	\begin{theorem}\label{main_res}
		For
		\begin{gather*}
			s_{\mathbb{L}} =\min\{r_{\mathbb L}, k+1 \},  \;\; 	s_{\bm y} =\min\{r_{\bm y}, k+2 \},\;\;
			s_{\mathbb{G}} =\min\{r_{\mathbb G}, k+1 \},  \;\; 	s_{\bm z} =\min\{r_{\bm z}, k+2 \},
		\end{gather*}
		if the regularity assumption \eqref{reg_assumption} holds we have
		\begin{align*}
			\norm {u - u_h}_{\mathcal E_h^\partial}  &\lesssim h^{s_{\mathbb L}+\frac 1 2}\norm{\mathbb L}_{s_{\mathbb L},\Omega}  +  h^{s_{\bm y}-\frac 1 2}\norm{\bm y}_{s_{\bm y},\Omega} + h^{s_{\mathbb G}-\frac 1 2}\norm{\mathbb G}_{s_{\mathbb G},\Omega}{+h^{s_{\bm z}-\frac32}\norm{\bm z}_{s_{\bm z},\Omega}},\\
			\norm {\bm y - \bm y_h}_{\mathcal T_h}  &\lesssim    h^{s_{\mathbb L}+\frac 1 2}\norm{\mathbb L}_{s_{\mathbb L},\Omega}  +  h^{s_{\bm y}-\frac 1 2}\norm{\bm y}_{s_{\bm y},\Omega} + h^{s_{\mathbb G}-\frac 1 2}\norm{\mathbb G}_{s_{\mathbb G},\Omega}{+h^{s_{\bm z}-\frac32}\norm{\bm z}_{s_{\bm z},\Omega}},\\
			\norm {\mathbb G - \mathbb G_h}_{\mathcal T_h} &\lesssim  h^{s_{\mathbb L}+\frac 1 2}\norm{\mathbb L}_{s_{\mathbb L},\Omega}  +  h^{s_{\bm y}-\frac 1 2}\norm{\bm y}_{s_{\bm y},\Omega} + h^{s_{\mathbb G}-\frac 1 2}\norm{\mathbb G}_{s_{\mathbb G},\Omega}{+h^{s_{\bm z}-\frac32}\norm{\bm z}_{s_{\bm z},\Omega}},\\
			\norm {\bm z - \bm z_h}_{\mathcal T_h} &\lesssim   h^{s_{\mathbb L}+\frac 1 2}\norm{\mathbb L}_{s_{\mathbb L},\Omega}  +  h^{s_{\bm y}-\frac 1 2}\norm{\bm y}_{s_{\bm y},\Omega} + h^{s_{\mathbb G}-\frac 1 2}\norm{\mathbb G}_{s_{\mathbb G},\Omega}+  h^{s_{\bm z}-\frac32}\norm{\bm z}_{s_{\bm z},\Omega}.
		\end{align*}
		If $k \geq  1$, then
		\begin{align*}
			\norm {\mathbb L - \mathbb L_h}_{\mathcal T_h}   \lesssim h^{s_{\mathbb L}}\norm{\mathbb L}_{s_{\mathbb L},\Omega} + h^{s_{\bm y}-1}\norm{\bm y}_{s_{\bm y},\Omega} +h^{s_{\mathbb G}-1}\norm{\mathbb G}_{s_{\mathbb G},\Omega}{+h^{s_{\bm z}-2}\norm{\bm z}_{s_{\bm z},\Omega}}.
		\end{align*}
	\end{theorem}
	\begin{remark}
		The error estimates in Theorem \ref{main_res} are independent of the pressures $p$ and $q$, which are different {from} the error estimates in \cite[Theorem 4.1]{GongStokes_Tangential1}. Therefore, our HDG method is pressure-robust. We note that the HDG method considered  here has more degrees of freedom than {that} in \cite{GongStokes_Tangential1}, since we have introduced numerical traces for the  pressures.  We also note that the technique used in \cite{GongStokes_Tangential1} cannot be applied here to treat the case when $r_{\mathbb L}\le 1/2$. This low regularity for $\mathbb L = \nabla \bm y$ may appear when $\xi \leq 3/2$, which corresponds to a value of $\omega$ greater than $\omega_{3/2}\approx 0.839138753489667\pi$; see more details in Remark \ref{whyhighregualirty}.
		Moreover, the meshes here are restricted to be triangular, while in \cite{GongStokes_Tangential1} we can use general polygonal meshes.
	\end{remark}

	\textcolor{black}{Noticing that for $ \omega \in [\pi/3, \omega_{3/2}) $ we have that $\xi\in (3/2,4]$, the application of } Theorems \ref{main_res} and {\ref{regularity_res}} gives the following result.
	\begin{corollary}\label{cor:main_result}
		Suppose  $ \bm y_d \in \bm H^{\xi}(\Omega) $.  Let $ \omega \in [\pi/3, \omega_{3/2}) $ be the largest interior angle of $\Gamma$, and define $ r_{\Omega}$ by%, $ t_\Omega $ by
		\[
		r_{\Omega} = \min\left\{ \frac{3}{2}, \xi- \frac{1}{2} \right\} \in \textcolor{black}{(1, \frac{3}{2}]}.%,  \quad  t_{\Omega} = \min\left\{ t^* + 2, 1 + \frac{\pi}{\omega} \right\} \in (5/2, 3).
		\]
		Then the regularity condition \eqref{reg_assumption} is satisfied.  Also, if $ k \ge 1 $, then for any $ r < r_\Omega $ we have
		\begin{gather*}
			h^{\frac 1 2}\|\mathbb L - \mathbb L_h\|_{\mathcal T_h}  + \|\bm y  - \bm y_h\|_{\mathcal T_h} + 	\|\mathbb G - \mathbb G_h\|_{\mathcal T_h} + \|\bm z - \bm z_h\|_{\mathcal T_h} + \norm{u-u_h}_{\mathcal E_h^\partial}\lesssim h^{r}.
		\end{gather*}
		Moreover, if $k=0$, we have
		\begin{align*}
			\norm{u-u_h}_{\mathcal E_h^\partial}  + \|\bm y  - \bm y_h\|_{\mathcal T_h} + \|\bm z  - \bm z_h\|_{\mathcal T_h} + \|\mathbb G  - \mathbb G_h\|_{\mathcal T_h}  \lesssim h^{1/2}.
		\end{align*}
	\end{corollary}

	\subsection{Preliminary material}
	\label{sec:Projectionoperator}
		We use  the standard $L^2$ {projections} $\bm \Pi _{\mathbb K} : \mathbb L^2(\Omega)\to \mathbb {K}_h$, $\bm \Pi_V :  \bm L^2(\Omega) \to \bm{V} _h$, and $\Pi_W :  L^2(\Omega) \to W_h$ satisfying
	\begin{subequations}\label{projection_operator}
		\begin{align}
			(\bm{\Pi}_{\mathbb K}\mathbb L,\mathbb T)_K&=(\mathbb L,\mathbb T)_K  \qquad \forall \ \mathbb T \in[\mathcal P^{k}(K)]^{2\times 2},\label{projection_operator_a}\\
			(\bm\Pi_V\bm y,\bm v)_K&=(\bm y,\bm v)_K \qquad \forall  \ \bm v \in[\mathcal P^{k+1}(K)]^2,\label{projection_operator_b}\\
			(\Pi_Wp, w)_K&=(p, w)_K \qquad \forall  \ w\in \mathcal P^{k}(K).\label{projection_operator_c}
		\end{align}
	\end{subequations}
	For all edges $e$ of the triangle $K$, we also need the $L^2$-orthogonal projections {$ P_M$ onto $ M_h$, $ P_Q$ onto $ Q_h$,}  and $\bm P_M$ onto $\bm M_h$ satisfying
	\begin{subequations}
		\begin{align}
			{ \langle  P_M  u- u,  \mu\rangle_e} &= 0 \quad \forall \mu\in  M_h,\\
			{\langle  P_Q  p- p,  \mu\rangle_e} &= 0 \quad \forall \mu\in  Q_h,\\
			\langle \bm P_M \bm y-\bm y, \bm \mu\rangle_e &= 0 \quad \forall \bm\mu\in \bm M_h.\label{def_P_M}
		\end{align}
	\end{subequations}
	In the analysis, we use the following classical results \cite[Section 4.2]{HuShenSinglerZhangZheng_HDG_Dirichlet_control1}:
	\begin{subequations}
		\begin{align}
			\|\bm\Pi_{\mathbb K}\mathbb L-\mathbb L\|_{\mathcal T_h}&\lesssim h^{s_{\mathbb L}} \|\mathbb L\|_{{s_{\mathbb L}},\Omega}, \quad 		\|\bm\Pi_{\bm V}\bm y -\bm y\|_{\mathcal T_h} \lesssim h^{s_{\bm y}} \|\bm y\|_{{s_{\bm y}},\Omega},\\
			\|\bm\Pi_{\mathbb K}\mathbb L-\mathbb L\|_{\partial\mathcal T_h}&\lesssim h^{s_{\mathbb L}-\frac 1 2 } \|\mathbb L\|_{{s_{\mathbb L}},\Omega}, \quad 		\|\bm\Pi_{\bm V}\bm y -\bm y\|_{\partial \mathcal T_h} \lesssim h^{s_{\bm y}-\frac 1 2} \|\bm y\|_{{s_{\bm y}},\Omega},\\
			\|\Pi_W p -p\|_{\mathcal T_h}&\lesssim h^{s_{p}} \|p\|_{{s_p},\Omega}, \quad 		\|\bm P_M \bm y -\bm y\|_{\partial \mathcal T_h} \lesssim h^{s_{\bm y}-\frac 1 2} \|\bm y\|_{{s_{\bm y}},\Omega},\\
			\| P_M  u - u\|_{\partial \mathcal T_h} &\lesssim h^{s_{\bm y}-\frac 1 2} \|\bm  y\|_{{s_{\bm y}},\Omega},\quad {\| P_Q  p - p\|_{\partial \mathcal T_h} \lesssim h^{s_{p}-\frac 1 2} \|p\|_{{s_{p}},\Omega}}.
		\end{align}
	\end{subequations}
	We have the same projection error bounds for $\mathbb G$, $\bm z$ and $q$.

	For the error analysis in this section, we need to introduce the {classical Raviart-Thomas (RT)} space:
	\begin{align*}
		\mathcal R^{k}(K)=[\mathcal{P}^k(K) ]^2+\bm x \mathcal{P}^k (K),
	\end{align*}
	and define the RT projection {$\Pi^{\textup{RT}}:\bm H^1(K)\rightarrow {\mathcal R^{k+1}(K)} $}
	\begin{subequations}\label{RT}
		\begin{align}
			\langle\Pi^{\textup{RT}} \bm v\cdot\bm n, w \rangle_{e}&=\langle \bm v\cdot\bm n,w \rangle_e \qquad \forall w \in {\mathcal{P}^{k+1}(e)},e\subset \partial K,\label{RT_1}\\
			(\Pi^{\textup{RT}} \bm v,\bm w)_{K}&=(\bm v,\bm w)_K \quad\quad\quad \forall \bm w\in {[\mathcal{P}^{k}(K)]^2}.\label{RT_2}
		\end{align}
	\end{subequations}
	
	We also need the  following classical results \cite[Theorem 3.1]{Duran2008}:
	\begin{align*}
		\|\Pi^{\textup{RT}} \bm y - \bm y\|_{\mathcal T_h} \lesssim h^{s_{\bm y}}\|\bm y\|_{{s_{\bm y}},\Omega},\qquad
		\|\Pi^{\textup{RT}} \bm y - \bm y\|_{\partial \mathcal T_h} \lesssim h^{s_{\bm y}-1/2}\|\bm y\|_{{s_{\bm y}},\Omega}.
	\end{align*}
	
	By the well-known commutative diagram  \cite[Equation (38)]{Duran2008}  we have
	\begin{align*}
		{\nabla \cdot (\Pi^{\textup{RT}} \bm v) =  \Pi (\nabla \cdot \bm v)},
	\end{align*}
	where $\Pi$ is the standard $L^2$ projection from $L^2(K)$ onto $\mathcal P^{k+1}(K)$.  If $\bm v\in \bm H(\textup{div};\Omega)$ and $\nabla\cdot\bm v=0$, then
	$$
	\nabla \cdot (\Pi^{\textup{RT}} \bm v) = 0.
	$$
	{Applying} \cite[Lemma 3.1]{Duran2008} we have the following lemma.
	\begin{lemma}\label{Div_free}
		For any $\bm v\in \bm H(\textup{div};\Omega)$ and $\nabla\cdot\bm v=0$, we have $\Pi^{\textup{RT}}\bm v\in \bm V_h$.
	\end{lemma}
	
	To simplify notation, we define an HDG operator $ \mathscr B$. For all $(\mathbb L_h,\bm y_h,p_h,\widehat{p}_h, \widehat{\bm y}_h^o) \in \mathbb K_h\times\bm{V}_h\times W_h^0\times Q_h\times  \bm M_h(o)$, we define
	\begin{align}\label{def_B1}
%		\begin{split}
			\hspace{1em}&\hspace{-1em}  \mathscr B(\mathbb L_h,\bm y_h,p_h,\widehat{p}_h,\widehat{\bm y}_h^o; \mathbb T_1,\bm v_1, w_1,\widehat{w}_1,\bm \mu_1)\nonumber\\
			&=(\mathbb L_h,\mathbb T_1)_{\mathcal{T}_h}+(\bm y_h,\nabla\cdot \mathbb T_1)_{\mathcal{T}_h}-\left\langle \widehat{\bm y}_h^o, \mathbb T_1 \bm{n}\right\rangle_{\partial\mathcal{T}_h\backslash \mathcal E_h^\partial}-(\nabla\cdot \mathbb{L}_h,\bm v_1)_{\mathcal{T}_h}\nonumber\\
			& \quad -( p_h,\nabla\cdot \bm v_1)_{\mathcal{T}_h}+\langle \widehat{p}_h,\bm v_1\cdot\bm n \rangle_{\partial\mathcal{T}_h} +\langle h^{-1} P_{\bm M} \bm y_h, \bm v_1\rangle_{\partial {\mathcal{T}_h}}-\langle h^{-1}\widehat{\bm y}_h^o,\bm v_1 \rangle_{\partial \mathcal{T}_h\backslash\mathcal E_h^\partial}\nonumber\\
			&\quad +(\nabla\cdot\bm y_h, w_1)_{\mathcal T_h}-\langle \bm y_h\cdot\bm n,\widehat{w}_1\rangle_{\partial\mathcal T_h}+\langle {\mathbb L}_h \bm n - h^{-1} (\bm P_M\bm y_h - \widehat{\bm y}_h^o),\bm\mu_1\rangle_{\partial\mathcal T_h\backslash\mathcal E_h^\partial}
%		\end{split}	
	\end{align}
	for all $(\mathbb{T}_1,\bm v_1,w_1,\widehat{w}_1, \bm\mu_1)\in \mathbb K_h\times\bm{V}_h\times W_h^0\times Q_h\times \bm M_h(o)$.

	By the definition of $\mathscr B$,  we can rewrite the HDG formulation \eqref{HDG_discrete2} as follows: find $(\mathbb L_h,\bm y_h,p_h,\widehat{p}_h,\widehat{\bm y}_h^o;\mathbb G_h,\bm z_h,q_h,\widehat{q}_h,\widehat{\bm z}_h^o) \in [\mathbb K_h\times\bm{V}_h\times W_h^0\times Q_h\times  \bm M_h(o)]^2$ and $u_h\in M_h$  such that
	\begin{subequations}\label{HDG_full_discrete}
		\begin{align}
			\mathscr  B(\mathbb L_h,\bm y_h,p_h,\widehat{p}_h,\widehat{\bm y}_h^o; \mathbb T_1,\bm v_1, w_1,\widehat{w}_1,\bm \mu_1)&=\langle  u_h\bm{\tau}, \mathbb T_1 \bm{n} +h^{-1}\bm v_1\rangle_{\mathcal E_h^\partial}+(\bm f,\bm v_1)_{\mathcal T_h},\label{HDG_full_discrete_a}\\
			\mathscr   B(\mathbb G_h,\bm z_h,-q_h,-\widehat{q}_h,\widehat{\bm z}_h^o; \mathbb T_2,\bm v_2, w_2,\widehat{w}_2,\bm \mu_2)&=(\bm y_h - \bm y_d,\bm v_2)_{\mathcal T_h},\label{HDG_full_discrete_b}\\
			{\langle \mathbb G_h\bm n-h^{-1} \bm P_{M}\bm z_h,  \mu_3\bm \tau \rangle_{\mathcal E_h^\partial} }&=\gamma \langle   u_h,  \mu_3 \rangle_{\mathcal E_h^\partial}\label{HDG_full_discrete_e}
		\end{align}
	\end{subequations}
	for all $(\mathbb T_1,\bm v_1,w_1,\widehat{w}_1,\bm \mu_1;\mathbb T_2,\bm v_2,w_2,\widehat{w}_2,\bm \mu_2) \in [\mathbb K_h\times\bm{V}_h\times W_h^0\times Q_h\times  \bm M_h(o)]^2$ and $\mu_3\in M_h$.
	
	\begin{lemma}\label{property_B}
		For any $(\mathbb{L}_h,\bm y_h,p_h,\widehat{p}_h,\widehat{\bm y}_h^o)\in \mathbb{K}_h\times \bm V_h\times W_h  \times Q_h \times \bm M_h(o)$,
			\begin{align}\label{pro_B1}
			\begin{split}
				\hspace{1em}&\hspace{-1em}  \mathscr B(\mathbb L_h,\bm y_h,p_h,\widehat{p}_h,\widehat{\bm y}_h^o;\mathbb L_h,\bm y_h,p_h,\widehat{p}_h,\widehat{\bm y}_h^o)\\
				&=\|\mathbb L_h\|^2_{\mathcal T_h}+h^{-1}\|\bm P_{M} \bm y_h-\widehat{\bm y}_h^o \|^2_{\partial\mathcal T_h\backslash\mathcal E_h^\partial} +h^{-1}\|\bm P_{M} \bm y_h\|^2_{\mathcal E_h^\partial}.
			\end{split}
		\end{align}
	\end{lemma}
	\begin{proof}
		According to the definition of $\mathscr B$ in \eqref{def_B1} and integration by parts, we get
		\begin{align*}
			\hspace{1em}&\hspace{-1em} \mathscr B(\mathbb L_h,\bm y_h,p_h,\widehat{p}_h,\widehat{\bm y}_h^o;\mathbb L_h,\bm y_h,p_h,\widehat{p}_h,\widehat{\bm y}_h^o)\\
			&=(\mathbb L_h,\mathbb L_h)_{\mathcal{T}_h}+(\bm y_h,\nabla\cdot\mathbb L_h)_{\mathcal{T}_h}-\left\langle \widehat{\bm y}_h^o, \mathbb L_h \bm{n}\right\rangle_{\partial\mathcal{T}_h\backslash \mathcal E_h^\partial}-(\nabla\cdot\mathbb L_h,\bm y_h)_{\mathcal T_h}\\
			&\quad -( p_h,\nabla\cdot \bm y_h)_{\mathcal{T}_h} +\langle \widehat{p}_h,\bm y_h\cdot\bm n \rangle_{\partial\mathcal{T}_h}+\left\langle h^{-1}(\bm P_{M}\bm y_h-\widehat{\bm y}_h^o), \bm y_h\right\rangle_{\partial {\mathcal{T}_h}\backslash\mathcal E_h^\partial} \\
			&\quad+\langle  h^{-1} \bm P_{M} \bm y_h,\bm y_h\rangle_{\mathcal E_h^\partial} +(\nabla\cdot\bm y_h, p_h)_{\mathcal T_h} -\langle \bm y_h\cdot\bm n,\widehat{p}_h\rangle_{\partial\mathcal T_h}\\
			&\quad+\langle \mathbb L_h\bm n-h^{-1}(\bm P_{M}\bm y_h-\widehat{\bm y}_h^o),\widehat{\bm y}_h^o \rangle_{\partial\mathcal T_h\backslash\mathcal E^{\partial}_h}\\
			&=\|\mathbb L_h\|^2_{\mathcal T_h}+h^{-1}\|\bm P_{M} \bm y_h-\widehat{\bm y}_h^o\|^2_{\partial\mathcal T_h\backslash \mathcal E_h^\partial}+h^{-1} \|\bm P_{M} \bm y_h\|_{\mathcal E_h^\partial}^2.
		\end{align*}
	\end{proof}

	Similarly, 	for any $(\mathbb{G}_h,\bm z_h,q_h,\widehat{q}_h,\widehat{\bm z}_h^o)\in \mathbb{K}_h\times \bm V_h\times W_h  \times Q_h \times \bm M_h(o)$, we have
	\begin{align}\label{pro_B2}
	\begin{split}
	\hspace{1em}&\hspace{-1em}  \mathscr B(\mathbb G_h,\bm z_h,-q_h,-\widehat{q}_h,\widehat{\bm z}_h^o;\mathbb G_h,\bm z_h,-q_h,-\widehat{q}_h,\widehat{\bm z}_h^o)\\
	&=\|\mathbb G_h\|^2_{\mathcal T_h}+h^{-1}\|\bm P_{M} \bm z_h-\widehat{\bm z}_h^o \|^2_{\partial\mathcal T_h\backslash\mathcal E_h^\partial} +h^{-1}\|\bm P_{M} \bm z_h\|^2_{\mathcal E_h^\partial}.
	\end{split}
	\end{align}

	Next we give a property of $\mathscr B$ that is critically important to our error analysis of this method.
	\begin{lemma}\label{identical_equal}
		For any $(\mathbb{L}_h,\bm y_h,p_h,\widehat{p}_h,\widehat{\bm y}_h^o;\mathbb{G}_h,\bm z_h,q_h,\widehat{q}_h,\widehat{\bm z}_h^o)\in [\mathbb{K}_h\times \bm V_h\times W_h  \times Q_h \times \bm M_h(o)]^2$,
		\begin{align*}
			\mathscr B (\mathbb{L}_h,\bm y_h,p_h,\widehat{p}_h,\widehat{\bm y}_h^o;-\mathbb{G}_h, \bm z_h,q_h,\widehat{q}_h,\widehat{\bm z}_h^o)
			=\mathscr B (\mathbb{G}_h, \bm z_h,-q_h,-\widehat{q}_h,\widehat{\bm z}_h^o;-\mathbb{L}_h,\bm y_h,p_h,\widehat{p}_h,\widehat{\bm y}_h^o).
		\end{align*}
	\end{lemma}
	\begin{proof}
		By the definition of $\mathscr B$ in \eqref{def_B1}  we have
		\begin{align*}
			\hspace{1em}&\hspace{-1em}\mathscr B (\mathbb{L}_h,\bm y_h,p_h,\widehat{p}_h,\widehat{\bm y}_h^o;-\mathbb{G}_h, \bm z_h,q_h,\widehat{q}_h,\widehat{\bm z}_h^o)\\
			&=-(\mathbb L_h,\mathbb G_h)_{\mathcal{T}_h}-(\bm y_h,\nabla\cdot\mathbb G_h)_{\mathcal{T}_h}+\left\langle \widehat{\bm y}_h^o, \mathbb G_h \bm{n}\right\rangle_{\partial\mathcal{T}_h\backslash \mathcal E_h^\partial}-(\nabla\cdot\mathbb L_h,\bm z_h)_{\mathcal T_h}\\
			&\quad-( p_h,\nabla\cdot \bm z_h)_{\mathcal{T}_h}+\langle \widehat{p}_h,\bm z_h\cdot\bm n \rangle_{\partial\mathcal{T}_h}+ \langle h^{-1} P_{\bm M}\bm y_h, \bm z_h\rangle_{\partial {\mathcal{T}_h}}-\langle h^{-1} \widehat{\bm y}_h^o,\bm z_h \rangle_{\partial \mathcal{T}_h\backslash\mathcal E_h^\partial} \\
			&\quad +(\nabla\cdot\bm y_h, q_h)_{\mathcal T_h} -\langle \bm y_h\cdot\bm n,\widehat{q}_h\rangle_{\partial\mathcal T_h}+\langle {\mathbb L}_h \bm n - h^{-1} (\bm P_M\bm y_h - \widehat{\bm y}_h^o),\widehat{\bm z}_h^o\rangle_{\partial\mathcal T_h\backslash\mathcal E^{\partial}_h}.
		\end{align*}
		Rearrange the terms above to get
		\begin{align*}
			\hspace{1em}&\hspace{-1em}\mathscr B (\mathbb{L}_h,\bm y_h,p_h,\widehat{p}_h,\widehat{\bm y}_h^o;-\mathbb{G}_h, \bm z_h,q_h,\widehat{q}_h,\widehat{\bm z}_h^o)\\
			&=-(\mathbb G_h, \mathbb L_h)_{\mathcal{T}_h}-(\bm z_h,\nabla\cdot\mathbb L_h)_{\mathcal{T}_h}+\left\langle \widehat{\bm z}_h^o, \mathbb L_h \bm{n}\right\rangle_{\partial\mathcal{T}_h\backslash \mathcal E_h^\partial}-(\nabla\cdot\mathbb G_h,\bm y_h)_{\mathcal T_h}\\
			&\quad+(q_h,\nabla\cdot\bm y_h)_{\mathcal T_h}-\langle \widehat{q}_h,\bm y_h\cdot\bm n \rangle_{\partial\mathcal{T}_h}+ \langle h^{-1} P_{\bm M}\bm z_h, \bm y_h\rangle_{\partial {\mathcal{T}_h}}-\langle h^{-1} \widehat{\bm z}_h^o,\bm y_h \rangle_{\partial \mathcal{T}_h\backslash\mathcal E_h^\partial} \\
			&\quad -(\nabla\cdot\bm z_h, p_h)_{\mathcal T_h} +\langle \bm z_h\cdot\bm n,\widehat{p}_h\rangle_{\partial\mathcal T_h}+\langle {\mathbb G}_h \bm n - h^{-1} (\bm P_M\bm z_h - \widehat{\bm z}_h^o),\widehat{\bm y}_h^o\rangle_{\partial\mathcal T_h\backslash\mathcal E^{\partial}_h}\\
			&=\mathscr B (\mathbb{G}_h, \bm z_h,-q_h,-\widehat{q}_h,\widehat{\bm z}_h^o;-\mathbb{L}_h,\bm y_h,p_h,\widehat{p}_h,\widehat{\bm y}_h^o),
		\end{align*}
		{where we used the fact that $\bm z_h\in {\bm H}(\rm div;\Omega)$ and $\nabla\cdot\bm z_h=0$ in  Proposition \ref{prof_divergence_free}.}
	\end{proof}

	To prove the uniqueness of solution of  the HDG formulation, we need to recall the following BDM projection.

	\begin{lemma} \cite[Equation (2.3)]{MR799685}\label{BDM_projection_lemma}
		For any $K \in \mathcal{T}_{h}$ and $\bm{v} \in\left[H^{1}(K)\right]^2$, there exists a unique $\Pi^{\textup{BDM}} \bm{v} \in [\mathcal{P}^{k+1} (K)]^2$ such that
		\begin{subequations}\label{BDM_projection}
			\begin{align}
				\left\langle\Pi^{\textup{BDM}} \bm{v} \cdot \boldsymbol{n}_{e}, w_{k+1}\right\rangle_{e}&=\left\langle\boldsymbol{v} \cdot \boldsymbol{n}_{e}, w_{k+1}\right\rangle_{e} & \forall w_{k+1} \in \mathcal P^{k+1}(e), e \in \partial K, \label{BDM_projection1}\\
				\left(\Pi^{\textup{BDM}} \bm{v}, \nabla p_{k}\right)_{K}&=\left(\boldsymbol{v}, \nabla p_{k}\right)_{K} & \forall p_{k} \in \mathcal P^{k}(K), \label{BDM_projection2}\\
				\left(\Pi^{\textup{BDM}} \bm{v}, \operatorname{curl}\left(b_{K} p_{k-1}\right)\right)_{K}&=\left(\boldsymbol{v},  \operatorname{curl}\left(b_{K} p_{k-1}\right)\right)_{K}, & \forall p_{k-1} \in \mathcal P^{k-1}(K), \label{BDM_projection3}
			\end{align}
		\end{subequations}	
		where {$b_K=\lambda_1\lambda_2\lambda_3$} is a ``bubble'' function and $\textup{curl} \phi = [\partial_y \phi, -\partial_x \phi]^\top$.  If $k =0$, then \eqref{BDM_projection3} is vacuous and $\Pi^{\textup{BDM}}$ is defined by \eqref{BDM_projection1} and \eqref{BDM_projection2}.
	\end{lemma}

\begin{remark}
	In \cite[Lemma 2.1]{MR799685}, Brezzi, Douglas and Marini proved that the system {\eqref{BDM_projection} determines} $\Pi^{\textup{BDM}}$ uniquely. In other words, the matrix formed from the left hand side of  \eqref{BDM_projection} is non-singular.  Hence, for any $z_1 \in H^1(e)$, $\bm z_2, \bm z_3 \in \left[H^{1}(K)\right]^2$, we can {uniquely} determine $\bm v_h \in [\mathcal{P}^{k+1} (K)]^2$ such that
	\begin{subequations}
		\begin{align}
		\left\langle \bm v_h \cdot \boldsymbol{n}_{e}, w_{k+1}\right\rangle_{e}&=\left\langle z_1, w_{k+1}\right\rangle_{e} & \forall w_{k+1} \in \mathcal P^{k+1}(e), e \in \partial K, \label{BDM_projection4}\\
		\left( \bm v_h, \nabla p_{k}\right)_{K}&=\left(\bm z_2, \nabla p_{k}\right)_{K} & \forall p_{k} \in \mathcal P^{k}(K), \label{BDM_projection5}\\
		\left( \bm v_h, \operatorname{curl}\left(b_{K} p_{k-1}\right)\right)_{K}&=\left(\bm z_3,  \operatorname{curl}\left(b_{K} p_{k-1}\right)\right)_{K}, & \forall p_{k-1} \in \mathcal P^{k-1}(K). \label{BDM_projection6}
		\end{align}
	\end{subequations}	
\end{remark}

	\begin{theorem}\label{exis_uni}
		There exists a unique solution of the HDG discrete optimality system \eqref{HDG_full_discrete}.
	\end{theorem}
	\begin{proof}
		Since the system \eqref{HDG_full_discrete} is finite dimensional, we only need to prove the uniqueness.  Therefore, we assume $\bm y_d = \bm f =0$ and we show the system \eqref{HDG_full_discrete} only has the trivial solution.
		
		First, take {$(\mathbb T_1,\bm v_1,w_1,\widehat{w}_1,\bm \mu_1) = (-\mathbb{G}_h,\bm z_h,-q_h,-\widehat{q}_h,\widehat {\bm z}_h^o)$}, $(\mathbb T_2,\bm v_2,w_2,\widehat{w}_2,\bm \mu_2) = \\ (-\mathbb{L}_h,\bm y_h,p_h,
		\widehat{p}_h,\widehat {\bm y}_h^o)$,  and $\mu_3 = - u_h $ in \eqref{HDG_full_discrete}, respectively.  By Lemma \ref{identical_equal} we have
		\begin{align*}
			&\mathscr B (\mathbb{L}_h,\bm y_h,p_h,\widehat{p}_h,\widehat{\bm y}_h^o;-\mathbb{G}_h, \bm z_h,-q_h,-\widehat{q}_h,\widehat{\bm z}_h^o)-\mathscr B (\mathbb{G}_h, \bm z_h,q_h,\widehat{q}_h,\widehat{\bm z}_h^o;-\mathbb{L}_h,\bm y_h,p_h,\widehat{p}_h,\widehat{\bm y}_h^o)  \\
			&=  {-(\bm y_h,\bm y_h)_{\mathcal T_h} -  \gamma \langle u_h, u_h\rangle_{\mathcal E_h^\partial}} \\
			&= 0.
		\end{align*}
		This implies $\bm y_h = u_h = 0$ since $\gamma>0$.
		
		Next, taking $(\mathbb T_1,\bm v_1,w_1,\widehat{w}_1,\bm \mu_1) = (\mathbb L_h,\bm y_h,p_h,\widehat{p}_h,\widehat{\bm y}_h^o)$ in \eqref{HDG_full_discrete_a} and $(\mathbb T_2,\bm v_2,w_2,\widehat{w}_2,\bm \mu_2) = (\mathbb G_h,\bm z_h,
		q_h, \widehat{q}_h, \widehat{\bm z}_h^o)$ in \eqref{HDG_full_discrete_b} and using  Lemma \ref{property_B}, we obtain  $\mathbb{L}_h= \mathbb{G}_h=  0, \widehat {\bm y}_h^o = \widehat {\bm z}_h^o=0$.
		
		Next,  taking $(\mathbb{T}_1,w_1,\widehat{w}_1,\bm \mu_1)=(0,0,0,0)$ and $(\mathbb{T}_2,\bm v_2,w_2,\widehat{w}_2,\bm \mu_2)=(0,0,0,0,0)$ and applying integration by parts gives
		\begin{align} \label{inter_exist_1}
			(\nabla p_h,\bm v_1)_{\mathcal{T}_h}+\langle \widehat{p}_h-p_h,\bm v_1\cdot\bm n \rangle_{\partial\mathcal{T}_h}=0.
		\end{align}
		
		Next, set $z_1 = \widehat p_h - p_h$ in \eqref{BDM_projection4}, $\bm z_2 = \bm 0$  in \eqref{BDM_projection5}, and $\bm z_3 = \bm 0$ in  \eqref{BDM_projection6}.  Then there  exists a unique  $\bm v_1\in [\mathcal{P}^{k+1} (K)]^2$ such that on each element $K$ we have
			\begin{align*}
			\left\langle \bm v_1 \cdot \boldsymbol{n}_{e}, w_{k+1}\right\rangle_{e}&=\left\langle  \widehat p_h - p_h, w_{k+1}\right\rangle_{e} & \forall w_{k+1} \in \mathcal P^{k+1}(e), e \in \partial K, \\
			\left( \bm v_1, \nabla p_{k}\right)_{K}&=0 & \forall p_{k} \in \mathcal P^{k}(K).
			\end{align*}
		This implies that  $(\bm v_1,\nabla p_h)_K=0$ and  $\bm v_1\cdot\bm n= \widehat{p}_h-p_h$ on $\partial K$. This gives $\widehat{p}_h=p_h$.
		
		Finally, {taking} $\bm v_1=\nabla p_h$ in \eqref{inter_exist_1} we have $\nabla p_h=0$, which together with {the fact that} $\widehat{p}_h$ is single-valued on each edge implies $p_h$ is a constant on the whole domain. Moreover, $p_h\in L_0^2(\Omega)$ gives $\widehat{p}_h=p_h=0$. Following the same idea gives $\widehat{q}_h=q_h=0$.
	\end{proof}

	\subsection{Proof of Theorem \ref{main_res}}
	We follow the strategy of our earlier work \cite{GongStokes_Tangential1} and split the proof into
	eight steps.  Consider the following auxiliary problem: find  $(\mathbb L_h(u),\bm y_h(u), p_h(u), \widehat p_h(u), \widehat {\bm y}_h^o(u); \\ \mathbb G_h(u),\bm z_h(u), q_h(u), \widehat q_h(u),\widehat {\bm z}_h^o(u))  \in [\mathbb K_h\times\bm{V}_h\times W_h^0 \times Q_h\times \bm M_h(o)]^2$  such that
	\begin{subequations}\label{HDG_inter_u}
		\begin{align}
			\mathscr B(\mathbb L_h( u),\bm y_h( u),p_h( u),\widehat p_h( u), \widehat{\bm y}_h^o( u);\mathbb T_1,\bm v_1, w_1,\widehat w_1,\bm \mu_1)&=\langle  (P_M u)\bm \tau, h^{-1} \bm v_1+\mathbb T_1\bm n \rangle_{\mathcal E_h^\partial}\nonumber\\
			& \quad+(\bm f,\bm v_1)_{\mathcal T_h}, \label{HDG_u_a}\\
			\mathscr B(\mathbb G_h( u),\bm z_h( u),-q_h( u),-\widehat q_h(u), \widehat{\bm z}_h^o( u);\mathbb T_2,\bm v_2, w_2,\widehat w_2,\bm \mu_2)&=(\bm y_h( u) - \bm y_d,\bm v_2)_{\mathcal T_h}\label{HDG_u_b}
		\end{align}
	\end{subequations}
	for all $(\mathbb T_1,\bm v_1,w_1,\widehat{w}_1,\bm \mu_1;\mathbb T_2,\bm v_2,w_2,\widehat{w}_2,\bm \mu_2) \in [\mathbb K_h\times\bm{V}_h\times W_h^0\times Q_h\times  \bm M_h(o)]^2$.
	
	We also note that although the proof strategy is very similar to \cite{GongStokes_Tangential1}, a simple rewriting of the proofs for the settings of this paper is not enough. For each of the  following lemmas,  we must take care of the spaces of velocity and pressure so that estimates are independent of the pressure.
	
	We begin by bounding the error between the solutions of the auxiliary problem and the mixed form \eqref{TCOSE1}-\eqref{TCOOC} of the optimality system.  Define
	\begin{equation}\label{notation_1}
		\begin{split}
			\delta^{\mathbb{L}} &=\mathbb{L}-{\bm\Pi}_{\mathbb{K}}\mathbb{L},  \qquad\qquad\qquad \qquad\qquad\qquad\varepsilon^{\mathbb{L}}_h={\bm\Pi}_{\mathbb{K}}\mathbb{L}-\mathbb{L}_h(u),\\
			\delta^{\bm y}&=\bm y- \Pi^{\textup{RT}} \bm y, \qquad\qquad \qquad\qquad\qquad\quad \ \  \varepsilon^{\bm y}_h = \Pi^{\textup{RT}}\bm y-\bm y_h(u),\\
			\delta^{p} &= p- \Pi_W p,  \qquad\qquad\qquad \qquad\qquad\qquad \varepsilon^p_h= \Pi_W p-p_h(u),\\
			\delta^{\widehat{p}}&={p-P_Q p}, \qquad\qquad\qquad \qquad\qquad\qquad \  \varepsilon^{\widehat{p}}_h= {P_Q p-\widehat{p}_h( u)},\\
			\delta^{\widehat {\bm y}} &= \bm y-\bm P_{ M} \bm y,  \qquad\qquad\qquad \qquad\qquad\qquad  \varepsilon^{\widehat {\bm y}}_h=\bm P_{ M} \bm y-\widehat{\bm y}_h(u),
		\end{split}
	\end{equation}
	where $\widehat {\bm y}_h( u) = \widehat {\bm y}_h^o( u)$ on $\mathcal E_h^o$ and $\widehat {\bm y}_h( u) =  (P_M  u)\bm{\tau}$ on $\mathcal E_h^{\partial}$, then $\varepsilon_h^{\widehat {\bm y}} = \bm 0$ on $\mathcal E_h^{\partial}$.
	
	\subsubsection*{Step 1: The error equation for part 1 of the auxiliary problem \eqref{HDG_u_a}}
	\begin{lemma}\label{projection_error_lyp}
		Let $(\mathbb L, \bm y, p)$ be the solution of the optimality system \eqref{optimality_system}. Then we have for all $(\mathbb T_1,\bm v_1,w_1, \widehat{w}_1,\bm \mu_1) \in \mathbb K_h\times\bm{V}_h\times W_h^0\times Q_h\times  \bm M_h(o)$ that
		\begin{align*}
			\hspace{1em}&\hspace{-1em}  \mathscr B(\bm \Pi_{\mathbb{K}}\mathbb L, \Pi^{\textup{RT}} \bm y,\Pi_W p, P_Q p ,\bm P_{M} \bm y; \mathbb T_1,\bm v_1, w_1,\widehat{w}_1,\bm \mu_1)\\
			&=(\bm f,\bm v_1)_{\mathcal T_h} + \langle (P_M u)\bm \tau,\mathbb{T}_1\bm n+h^{-1}\bm v_1\rangle_{\mathcal E_h^\partial}-\langle h^{-1}\bm P_{M} \delta^{\bm y},\bm v_1 \rangle_{\partial\mathcal T_h} \\
			&\quad +\langle \delta^{\mathbb{L}} \bm n, \bm v_1\rangle_{\partial\mathcal T_h} -\langle \delta^{\mathbb{L}} \bm n, \bm \mu_1 \rangle_{\partial\mathcal T_h\backslash\mathcal E_h^{\partial}}+ \langle h^{-1}\bm P_{M} \delta^{\bm y},\bm \mu_1 \rangle_{\partial\mathcal T_h\backslash\mathcal E_h^{\partial}}.
		\end{align*}
	\end{lemma}
	\begin{proof}
		Since $\nabla \cdot \bm y = 0$, by Lemma \ref{Div_free} we have $ \Pi^{\textup{RT}} \bm y \in \bm V_h$.
		By the definition of the operator $\mathscr B$ in \eqref{def_B1} we obtain
		\begin{align*}
			\hspace{1em}&\hspace{-1em}  \mathscr B(\bm \Pi_{\mathbb{K}}\mathbb L, \Pi^{\textup{RT}} \bm y,\Pi_W p, {P_Q p},\bm P_{M} \bm y; \mathbb T_1,\bm v_1, w_1,\widehat{w}_1,\bm \mu_1)\nonumber\\
			&=(\bm \Pi_{\mathbb{K}}\mathbb L,\mathbb T_1)_{\mathcal{T}_h}+(\Pi^{\textup{RT}}\bm y,\nabla\cdot \mathbb T_1)_{\mathcal{T}_h}-\langle \bm P_{M} \bm y, \mathbb T_1 \bm{n}\rangle_{\partial\mathcal{T}_h\backslash \mathcal E_h^\partial}-(\nabla\cdot \Pi_{\mathbb{K}}\mathbb{L},\bm v_1)_{\mathcal{T}_h}\nonumber\\
			&\quad-( \Pi_{W} p,\nabla\cdot \bm v_1)_{\mathcal{T}_h}+\langle{ P_Q p},\bm v_1\cdot\bm n \rangle_{\partial\mathcal{T}_h} +\langle h^{-1} \bm P_{ M} \Pi^{\textup{RT}}\bm y, \bm v_1\rangle_{\partial {\mathcal{T}_h}}\nonumber\\
			&\quad-\langle h^{-1}\bm P_{M}\bm y,\bm v_1 \rangle_{\partial \mathcal{T}_h\backslash\mathcal E_h^\partial} +(\nabla\cdot \Pi^{\textup{RT}} \bm y, w_1)_{\mathcal T_h}-\langle \Pi^{\textup{RT}} \bm y\cdot\bm n,\widehat{w}_1\rangle_{\partial\mathcal T_h}\\
			&\quad+\langle \Pi_{\mathbb{K}}\mathbb{L}\bm n-h^{-1}(\bm P_{ M}\Pi^{\textup{RT}}\bm y-\bm P_{M}\bm y),\bm \mu_1\rangle_{\partial\mathcal T_h\backslash\mathcal E_h^{\partial}}.
		\end{align*}
		By definition of the $L^2$ projections  and the RT projection, we have
		\begin{align*}
			\hspace{1em}&\hspace{-1em}  \mathscr B(\bm \Pi_{\mathbb{K}}\mathbb L, \Pi^{\textup{RT}} \bm y,\Pi_W p, P_M p ,\bm P_{M} \bm y; \mathbb T_1,\bm v_1, w_1,\widehat{w}_1,\bm \mu_1)\\
			&=(\mathbb L,\mathbb T_1)_{\mathcal{T}_h}+(\bm y,\nabla\cdot \mathbb T_1)_{\mathcal{T}_h}-\left\langle  \bm y, \mathbb T_1 \bm{n}\right\rangle_{\partial\mathcal{T}_h\backslash \mathcal E_h^\partial}+(\nabla\cdot\delta^{\mathbb{L}},\bm v_1)_{\mathcal{T}_h} \nonumber\\
			&\quad-(\nabla\cdot \mathbb{L},\bm v_1)_{\mathcal{T}_h}-(p,\nabla\cdot \bm v_1)_{\mathcal{T}_h}+\langle  p,\bm v_1\cdot\bm n \rangle_{\partial\mathcal{T}_h} +\langle h^{-1}\bm P_{M}\bm y,\bm v_1 \rangle_{\partial\mathcal T_h}\nonumber\\
			&\quad-\langle h^{-1}\bm P_{M} \delta^{\bm y},\bm v_1 \rangle_{\partial\mathcal T_h} -\langle h^{-1}\bm P_{M}\bm y,\bm v_1 \rangle_{\partial\mathcal T_h\backslash\mathcal E_h^\partial} +(\nabla\cdot \Pi^{\textup{RT}} \bm y, w_1)_{\mathcal T_h}\\
			&\quad-\langle \bm y\cdot\bm n,\widehat{w}_1\rangle_{\partial\mathcal T_h} +\langle {\bm\Pi}_{\mathbb K} {\mathbb L}\bm n,\bm \mu_1 \rangle_{\partial\mathcal T_h\backslash\mathcal E_h^{\partial}} + \langle h^{-1}\bm P_{M} \delta^{\bm y},\bm \mu_1 \rangle_{\partial\mathcal T_h\backslash\mathcal E_h^{\partial}}.
		\end{align*}
		Moreover, integration by parts gives
		\begin{align*}
			(\nabla\cdot \Pi^{\textup{RT}} \bm y, w_1)_{\mathcal T_h} &= \langle \Pi^{\textup{RT}} \bm y \cdot \bm n, w_1 \rangle_{\partial\mathcal T_h} - 	(\Pi^{\textup{RT}} \bm y, \nabla w_1)_{\mathcal T_h} \\
			&= \langle \bm y \cdot \bm n, w_1 \rangle_{\partial\mathcal T_h} - 	(\bm y, \nabla w_1)_{\mathcal T_h} \\
			&= 	(\nabla\cdot \bm y,  w_1)_{\mathcal T_h} \\
			&=0.
		\end{align*}
		Note that the exact solutions $\mathbb{L}$, $\bm y$ and $p$ satisfy
		\begin{align*}
			(\mathbb L,\mathbb T_1)_{\mathcal{T}_h}+(\bm y,\nabla\cdot\mathbb T_1)_{\mathcal{T}_h}-\langle\bm y,\mathbb T_1\bm n\rangle_{\partial \mathcal{T}_h\backslash\mathcal E_h^\partial}&=\langle  u\bm \tau,\mathbb{T}_1 \bm n \rangle_{\mathcal E_h^\partial},\\
			-(\nabla\cdot(\mathbb{L} - p\mathbb I),  \bm v_1)_{\mathcal{T}_h}&= (\bm f,\bm v_1)_{\mathcal T_h},  \\
			(\nabla\cdot\bm y, w_1)_{\mathcal{T}_h}&=0,\\
			{\langle \bm y\cdot\bm n,\widehat{w}_1\rangle_{\partial\mathcal T_h}}&=0
		\end{align*}
		for all $(\mathbb T_1,\bm v_1,w_1,\widehat{w}_1)\in \mathbb{K}_h \times \bm V_h\times W_h^0\times Q_h$ and $\bm y=u\bm{\tau}$ on $\mathcal E_h^\partial$. Then we have
		\begin{align*}
			\hspace{1em}&\hspace{-1em}  \mathscr B(\bm \Pi_{\mathbb{K}}\mathbb L, \Pi^{\textup{RT}} \bm y,\Pi_W p, P_M p ,\bm P_{M} \bm y; \mathbb T_1,\bm v_1, w_1,\widehat{w}_1,\bm \mu_1)\\
			&=(\bm f,\bm v_1)_{\mathcal T_h} +\langle (P_M u)\bm \tau,\mathbb{T}_1\bm n+h^{-1}\bm v_1\rangle_{\mathcal E_h^\partial}-\langle h^{-1}\bm P_{M} \delta^{\bm y},\bm v_1 \rangle_{\partial\mathcal T_h}\\
			&\quad +(\nabla\cdot\delta^{\mathbb{L}},\bm v_1)_{\mathcal{T}_h}+\langle \bm{\Pi}_{\mathbb K} {\mathbb L}\bm n,\bm \mu_1 \rangle_{\partial\mathcal T_h\backslash\mathcal E_h^{\partial}} + \langle h^{-1}\bm P_{M} \delta^{\bm y},\bm \mu_1 \rangle_{\partial\mathcal T_h\backslash\mathcal E_h^{\partial}}.
		\end{align*}
		Since $\mathbb L\in \mathbb H^{r_{\mathbb L}}(\Omega)$ with $r_{\mathbb L}>1/2$, then $\langle \mathbb L\bm n, \bm \mu_1\rangle_{\partial\mathcal T_h\backslash\mathcal E_h^\partial} = 0$. This implies
		\begin{align*}
			\hspace{1em}&\hspace{-1em}  \mathscr B(\bm \Pi_{\mathbb{K}}\mathbb L, \Pi^{\textup{RT}} \bm y,\Pi_W p, P_M p ,\bm P_{M} \bm y; \mathbb T_1,\bm v_1, w_1,\widehat{w}_1,\bm \mu_1)\\
			&=(\bm f,\bm v_1)_{\mathcal T_h} + \langle (P_M u)\bm \tau,\mathbb{T}_1\bm n+h^{-1}\bm v_1\rangle_{\mathcal E_h^\partial}-\langle h^{-1}\bm P_{M} \delta^{\bm y},\bm v_1 \rangle_{\partial\mathcal T_h} \\
			&\quad +\langle \delta^{\mathbb{L}} \bm n, \bm v_1\rangle_{\partial\mathcal T_h} -\langle \delta^{\mathbb{L}} \bm n, \bm \mu_1 \rangle_{\partial\mathcal T_h\backslash\mathcal E_h^{\partial}}+ \langle h^{-1}\bm P_{M} \delta^{\bm y},\bm \mu_1 \rangle_{\partial\mathcal T_h\backslash\mathcal E_h^{\partial}},
		\end{align*}
		{where we used the fact that $(\mathbb{L}-\bm\Pi_{\mathbb{K}}\mathbb{L},\nabla\bm v_1)_{\mathcal T_h}=0$.}
	\end{proof}
	\begin{remark}\label{whyhighregualirty}
		In \cite{GongStokes_Tangential1}, we used $\mathbb L- p\mathbb I \in \mathbb H(\textup{div},\Omega)$ when $s_{\mathbb L}\le 1/2$. However, $\mathbb L \in \mathbb H(\textup{div},\Omega)$ does not hold here. Hence, we assume $r_{\mathbb L}>1/2$ so that $\mathbb L$ {has} a well-defined trace. Improving the analysis to handle the case $s_{\mathbb L}\le 1/2$ is left to be considered elsewhere.
	\end{remark}	
	
	Subtract part 1 of  \eqref{HDG_u_a} from Lemma \ref{projection_error_lyp} to obtain the following lemma.
	\begin{lemma}\label{LM:3.8}
		For all $(\mathbb T_1,\bm v_1,w_1,\widehat{w}_1,\bm \mu_1) \in \mathbb K_h\times\bm{V}_h\times W_h^0\times Q_h\times  \bm M_h(o)$,  we have
		\begin{align} \label{error_eq_lyp}
			\mathscr B(\varepsilon_h^{\mathbb{L}},\varepsilon^{\bm y}_h,\varepsilon^p_h,\varepsilon_h^{\widehat{p}},\varepsilon^{\widehat{\bm y}}_h;\mathbb T_1,\bm v_1,w_1,\widehat{w}_1,\bm\mu_1) &= -\langle h^{-1}\bm P_{M} \delta^{\bm y},\bm v_1 \rangle_{\partial\mathcal T_h}  +\langle h^{-1}\bm P_{M} \delta^{\bm y},\bm \mu_1 \rangle_{\partial\mathcal T_h\backslash \mathcal E_h^\partial}  \nonumber\\
			&\quad +\langle \delta^{\mathbb{L}} \bm n, \bm v_1\rangle_{\partial\mathcal T_h} -\langle \delta^{\mathbb{L}} \bm n, \bm \mu_1 \rangle_{\partial\mathcal T_h\backslash\mathcal E_h^{\partial}}.
		\end{align}
	\end{lemma}

	\subsubsection*{Step 2: Estimate for $\varepsilon_h^{\mathbb{L}}$}
	We first provide a key inequality which was {proven in \cite[Lemma 4.7]{GongStokes_Tangential1}}.
	\begin{lemma}\label{firstlemmastep1}
		We have
		\begin{equation} \label{grad_y_h}
			\|\nabla \varepsilon_h^{\bm y}\|_{\mathcal{T}_h} +h^{-\frac{1}{2}}\|\varepsilon_h^{\bm y} -\varepsilon_h^{\widehat{\bm y}}\|_{\partial \mathcal{T}_h} \lesssim \|\varepsilon_h^{\mathbb{L}}\|_{\mathcal{T}_h} +h^{-\frac{1}{2}} \|\bm P_{M} \varepsilon_h^{\bm y} -\varepsilon_h^{\widehat{\bm y}}\|_{\partial \mathcal{T}_h}.
		\end{equation}
	\end{lemma}
	\begin{lemma} \label{error_energy_norm}
		We have
		\begin{align*}
			\|\varepsilon_h^{\mathbb{L}}\|_{\mathcal{T}_h} +h^{-\frac{1}{2}} \|\bm P_{ M} \varepsilon_h^{\bm y} -\varepsilon_h^{\widehat{\bm y}}\|_{\partial \mathcal{T}_h}\lesssim h^{s_{\mathbb L}}\norm{\mathbb L}_{s_{\mathbb L},\Omega} + h^{s_{\bm y}-1}\norm{\bm y}_{s_{\bm y},\Omega}.
		\end{align*}
	\end{lemma}
	\begin{proof}
		First, since $\varepsilon_h^{\widehat{\bm y}}=0$ on $\mathcal E_h^\partial$, the basic property of $\mathscr B$ in Lemma \ref{property_B} gives
		\begin{align*}
			\mathscr B(\varepsilon_h^{\mathbb{L}},\varepsilon^{\bm y}_h,\varepsilon^p_h,\varepsilon_h^{\widehat{p}},\varepsilon^{\widehat{\bm y}}_h;\varepsilon_h^{\mathbb{L}},\varepsilon^{\bm y}_h,\varepsilon^p_h,\varepsilon_h^{\widehat{p}},\varepsilon^{\widehat{\bm y}}_h)=\|\varepsilon_h^{\mathbb{L}}\|^2_{\mathcal{T}_h} +h^{-1} \|\bm P_{ M} \varepsilon_h^{\bm y} -\varepsilon_h^{\widehat{\bm y}}\|^2_{\partial \mathcal{T}_h}.
		\end{align*}
		On the other hand, taking $(\mathbb T_1,\bm v_1,p_1,\widehat p_1,\bm \mu_1)=(\varepsilon_h^{\mathbb{L}},\varepsilon^{\bm y}_h,\varepsilon^p_h,\varepsilon^{\widehat p}_h,\varepsilon^{\widehat{\bm y}}_h)$ in \eqref{error_eq_lyp} gives
		\begin{align*}
			\|\varepsilon_h^{\mathbb{L}}\|^2_{\mathcal{T}_h} +h^{-1} \|\bm P_{ M} \varepsilon_h^{\bm y} -\varepsilon_h^{\widehat{\bm y}}\|^2_{\partial \mathcal{T}_h} = \langle \delta^{\mathbb{L}} \bm n,  \varepsilon_h^{\bm y}  - \varepsilon_h^{\widehat{\bm y}} \rangle_{\partial\mathcal T_h} - \langle h^{-1}\delta^{\bm y},\bm P_{M} \varepsilon_h^{\bm y} -\varepsilon_h^{\widehat{\bm y}} \rangle_{\partial\mathcal T_h}.
		\end{align*}
		By Lemma \ref{firstlemmastep1} and Young's inequality, we have
		\begin{align*}
			\|\varepsilon_h^{\mathbb{L}}\|_{\mathcal{T}_h} +h^{-\frac{1}{2}} \|\bm P_{ M} \varepsilon_h^{\bm y} -\varepsilon_h^{\widehat{\bm y}}\|_{\partial \mathcal{T}_h} \lesssim h^{s_{\mathbb L}}\norm{\mathbb L}_{s_{\mathbb L},\Omega} + h^{s_{\bm y}-1}\norm{\bm y}_{s_{\bm y},\Omega}.
		\end{align*}
	\end{proof}
	%\subsubsection*{Step 3: Estimate for $\varepsilon_h^p$}
	%\begin{lemma}\label{pressure_estimate_p_1}
	%	We have
	%	\begin{align}\label{error_po}
	%		\|{\varepsilon_h^p}\|_{\mathcal T_h} \lesssim h^{s_{\mathbb L}}\norm{\mathbb L}_{s_{\mathbb L},\Omega}+  h^{s_{p}}\norm{p}_{s^{p},\Omega} + h^{s_{\bm y}-1}\norm{\bm y}_{s_{\bm y},\Omega}.
	%	\end{align}
	%\end{lemma}
	\subsubsection*{Step 3: Estimate for $\varepsilon_h^{ y}$ by a duality argument}
	Next, we introduce the dual problem
	\begin{align} \label{dual_pde}
		\mathbb{A}-\nabla\bm \Phi=0~\text{in}\ \Omega,\;\;\;\;-\nabla\cdot\mathbb{A}-\nabla\Psi=\Theta \;\text{in}\ \Omega,\;\;\;\;
		\nabla\cdot\bm \Phi=0~\text{in}\ \Omega, \;\;\;\; \bm \Phi=0~\text{on}\ \partial\Omega.
	\end{align}
	Since the domain $\Omega$ is convex, we have the following regularity estimate:
	\begin{align}\label{regularity_dual}
		\|\mathbb{A}\|_{1,\Omega}+\|\bm \Phi\|_{2,\Omega}+\|\Psi\|_{1,\Omega}\le C\|\Theta\|_{0,\Omega}.
	\end{align}
	Before we estimate  $\varepsilon_h^{\bm y}$,  we introduce the following notation, which is similar to the earlier notation in \eqref{notation_1}:
	\begin{gather*}
		\delta^{\mathbb{A}} =\mathbb{A}-{\bm\Pi _{\mathbb K}} \mathbb{A},  \;\;\;    \delta^{\bm \Phi}=\bm \Phi-  \Pi^{\textup{RT}} \bm \Phi,  \;\;\;
		\delta^{\Psi}=\Psi - \Pi_W \Psi, \;\;\;
		\delta^{\widehat \Psi}=\Psi - P_Q \Psi, \;\;\;  \delta^{\widehat {\bm \Phi}} = \bm \Phi-\bm P_{M}\bm \Phi.
	\end{gather*}
	{Since $\bm\Phi=0$ on $\partial \Omega$, by using Lemma \ref{projection_error_lyp}} we have the following lemma:
	\begin{lemma}\label{projection_error_dual}
		Let $(\mathbb A, \bm \Phi, \Psi)$ be the {solution} of \eqref{dual_pde}, then for all $(\mathbb T_1,\bm v_1, w_1, \widehat{w}_1,\bm \mu_1) \in \mathbb K_h\times\bm{V}_h\times W_h^0\times Q_h\times  \bm M_h(o)$, we have
		\begin{align*}
			\hspace{1em}&\hspace{-1em}\mathscr B(\bm \Pi_{\mathbb{K}}\mathbb A, \Pi^{\textup{RT}} \bm \Phi,\Pi_W \Psi, P_Q \Psi ,\bm P_{M} \bm \Phi; \mathbb T_1,\bm v_1, w_1,\widehat{w}_1,\bm \mu_1)\\
			&=(\Theta,\bm v_1)_{\mathcal T_h}  -\langle h^{-1}\bm P_{M} \delta^{\bm \Phi},\bm v_1 \rangle_{\partial\mathcal T_h}  +\langle h^{-1}\bm P_{M} \delta^{\bm \Phi},\bm \mu_1 \rangle_{\partial\mathcal T_h\backslash \mathcal E_h^\partial}\\
			&\quad  + \langle \delta^{\mathbb{A}} \bm n, \bm v_1\rangle_{\partial\mathcal T_h} -\langle \delta^{\mathbb{A}} \bm n, \bm \mu_1 \rangle_{\partial\mathcal T_h\backslash\mathcal E_h^{\partial}}.
		\end{align*}
	\end{lemma}
	
	\begin{lemma}\label{error_y_lyp}
		We have
		\begin{equation} \label{error_yu}
			\|\varepsilon_h^{\bm y}\|_{\mathcal T_h} \lesssim h^{s_{\mathbb{L}}+1}\|\mathbb{L}\|_{s_{\mathbb{L}},\Omega} + h^{s_{\bm y}}\norm{\bm y}_{s_{\bm y},\Omega}.
		\end{equation}
	\end{lemma}
	\begin{proof}
		Consider the dual problem \eqref{dual_pde} and let $\Theta = \varepsilon_h^{\bm y}$. Since $\varepsilon_h^{\widehat{\bm y}}=0$ on $\mathcal E_h^{\partial}$, it follows from  Lemmas \ref{identical_equal} and \ref{projection_error_dual}  that
		\begin{align*}
			\hspace{1em}&\hspace{-1em} \mathscr B(\varepsilon_h^{\mathbb{L}},\varepsilon^{\bm y}_h,\varepsilon^p_h,\varepsilon_h^{\widehat{p}},\varepsilon^{\widehat{\bm y}}_h;-\bm \Pi_{\mathbb{K}} \mathbb{A}, \Pi^{\textup{RT}} \bm  \Phi, \Pi_W \Psi,P_Q \Psi,\bm P_{ M} \bm \Phi)\\
			&= \mathscr B(\bm \Pi_{\mathbb{K}} \mathbb{A}, \Pi^{\textup{RT}} \bm  \Phi, -\Pi_W \Psi,-P_Q \Psi,\bm P_{ M} \bm \Phi; -\varepsilon_h^{\mathbb{L}},\varepsilon^{\bm y}_h,\varepsilon^p_h,\varepsilon_h^{\widehat{p}},\varepsilon^{\widehat{\bm y}}_h)\\
			&= \langle \delta^{\mathbb{A}} \bm n, \varepsilon^{\bm y}_h - \varepsilon^{\widehat{\bm y}}_h\rangle_{\partial\mathcal T_h} - \langle h^{-1}\delta^{\bm \Phi},\bm P_{M} \varepsilon_h^{\bm y} -\varepsilon_h^{\widehat{\bm y}} \rangle_{\partial\mathcal T_h}+\|\varepsilon_h^{\bm y}\|_{\mathcal{T}_h}^2 .
		\end{align*}
		On the other hand, taking  $(\mathbb T_1,\bm v_1,w_1,\widehat{w}_1,\bm \mu_1) = ( -\bm \Pi_{\mathbb{K}} \mathbb{A}, \Pi^{\textup{RT}} \bm \Phi, \Pi_W \Psi,P_Q \Psi, \bm P_{ M} \bm \Phi)$ in \eqref{error_eq_lyp}  gives
		\begin{align*}
			\hspace{1em}&\hspace{-1em} \mathscr B(\varepsilon_h^{\mathbb{L}},\varepsilon^{\bm y}_h,\varepsilon^p_h,\varepsilon_h^{\widehat{p}},\varepsilon^{\widehat{\bm y}}_h;-\bm \Pi_{\mathbb{K}} \mathbb{A}, \Pi^{\textup{RT}} \bm  \Phi, \Pi_W \Psi,P_Q \Psi,\bm P_{ M} \bm \Phi) \\
			&= \langle \delta^{\mathbb{L}} \bm n,   \Pi^{\textup{RT}} \bm  \Phi-\bm P_M \bm \Phi \rangle_{\partial\mathcal T_h} +  \langle h^{-1}\delta^{\bm y},\bm P_{M} \delta^{\bm \Phi} \rangle_{\partial\mathcal T_h}.
		\end{align*}
		Then we have
		\begin{align*}
			\|\varepsilon_h^{\bm y}\|_{\mathcal{T}_h}^2 &= \langle \delta^{\mathbb{L}} \bm n,   \Pi^{\textup{RT}} \bm  \Phi-\bm P_M \bm \Phi \rangle_{\partial\mathcal T_h} - \langle \delta^{\mathbb{A}} \bm n, \varepsilon^{\bm y}_h - \varepsilon^{\widehat{\bm y}}_h\rangle_{\partial\mathcal T_h}\\
			& \quad + \langle h^{-1}\delta^{\bm \Phi},\bm P_{M} \varepsilon_h^{\bm y} -\varepsilon_h^{\widehat{\bm y}} \rangle_{\partial\mathcal T_h}+  \langle h^{-1}\delta^{\bm y},\bm P_{M} \delta^{\bm \Phi} \rangle_{\partial\mathcal T_h},
		\end{align*}
		{which together with the approximation properties of the $L^2$-orthogonal projection and the projection $\Pi^{\textup{RT}}$ and  Lemma \ref{error_energy_norm} gives} the desired result.
	\end{proof}

	As a consequence of Lemmas \ref{error_energy_norm} and \ref{error_y_lyp}, a simple application of the triangle inequality gives optimal convergence rates for $\|\mathbb L- \mathbb L_h( u)\|_{\mathcal T_h}$ and $\| \bm y - \bm y_h( u)\|_{\mathcal T_h}$:
	\begin{lemma}\label{lemma:step4_conv_rates}
		Let $(\mathbb L, \bm y, p)$ and $(\mathbb L_h(u), \bm y_h(u), p_h(u))$ be the {solution} of \eqref{optimality_system} and \eqref{HDG_u_a}, respectively. We have
		\begin{subequations}
			\begin{align}
				\|\mathbb{L} -\mathbb{L}_h( u)\|_{\mathcal{T}_h}&{\lesssim} h^{s_{\mathbb L}}\norm{\mathbb L}_{s_{\mathbb L},\Omega}+ h^{s_{\bm y}-1}\norm{\bm y}_{s_{\bm y},\Omega},\\
				\|\bm y-\bm y_h( u)\|_{\mathcal{T}_h}&\lesssim  h^{s_{\mathbb L}+1}\norm{\mathbb L}_{s_{\mathbb L},\Omega} + h^{s_{\bm y}}\norm{\bm y}_{s_{\bm y},\Omega}.
			\end{align}
		\end{subequations}
	\end{lemma}
	\subsubsection*{Step 4: The error equation for part 2 of the auxiliary problem \eqref{HDG_u_b}}
	We continue to bound the error between the solutions of the auxiliary problem and the mixed form \eqref{TCOSE1}-\eqref{TCOOC} of the optimality system.  In steps 4-5, we focus on the dual variables, i.e., $\mathbb{G}$, $\bm z$ and $q$. We use the following notation
	\begin{equation}\label{notation_3}
		\begin{split}
			\delta^{\mathbb{G}} &=\mathbb{G}-{\bm \Pi_{\mathbb{K}}} \mathbb{G},  \qquad\qquad \qquad\qquad\qquad\;\;\;\;\varepsilon^{\mathbb{G}}_h={\bm \Pi_{\mathbb{K}}} \mathbb{G}-\mathbb{G}_h( u),\\
			\delta^{\bm z}&=\bm z- {\Pi^{\textup{RT}}} \bm z, \qquad\qquad \qquad\qquad\qquad\;\;\;\;\;\varepsilon^{\bm z}_h=  {\Pi^{\textup{RT}}} \bm z-\bm z_h( u),\\
			\delta^q&=q- {\Pi_W} q, \qquad\qquad\qquad \qquad\qquad\;\;\;\;\;\; \;\varepsilon^{q}_h={\Pi_W} q - q_h( u),\\
			\delta^{\widehat{q}}&=q-P_Q q, \qquad\qquad\qquad \qquad\qquad\qquad \     \varepsilon^{\widehat{q}}_h= P_Q q-\widehat{q}_h( u),\\
			\delta^{\widehat {\bm z}} &= \bm z- \bm P_{ M} \bm z,  \qquad\qquad\qquad\qquad\qquad \;\;\;\;\;\;\; \varepsilon^{\widehat {\bm z}}_h=\bm P_{ M} \bm z-\widehat{\bm z}_h( u).
		\end{split}
	\end{equation}
	
	The derivation of the error equation for part 2 of the auxiliary problem \eqref{HDG_u_b} is similar to the analysis for part 1 of the auxiliary problem in step 1. Therefore, we state the result and omit the proof.
	
	\begin{lemma}
		For all $(\mathbb T_2,\bm v_2,w_2,\widehat{w}_2,\bm \mu_2) \in \mathbb K_h\times\bm{V}_h\times W_h^0\times Q_h\times  \bm M_h(o)$,  we have
		\begin{align} \label{error_eq_Gzq}
			\begin{split}
				\hspace{1em}&\hspace{-1em}\mathscr B(\varepsilon_h^{\mathbb{G}},\varepsilon^{\bm z}_h,-\varepsilon^q_h,-\varepsilon_h^{\widehat{q}},\varepsilon^{\widehat{\bm z}}_h;\mathbb T_2,\bm v_2,w_2,\widehat{w}_2,\bm\mu_2) \\
				&=-\langle h^{-1}\bm P_{M} \delta^{\bm z},\bm v_2 \rangle_{\partial\mathcal T_h}  +\langle h^{-1}\bm P_{M} \delta^{\bm z},\bm \mu_2 \rangle_{\partial\mathcal T_h\backslash \mathcal E_h^\partial} \\
				&\quad  + \langle \delta^{\mathbb{G}} \bm n, \bm v_2\rangle_{\partial\mathcal T_h} -\langle \delta^{\mathbb{G}} \bm n, \bm \mu_2 \rangle_{\partial\mathcal T_h\backslash\mathcal E_h^{\partial}} +(\bm y - \bm y_h(u), \bm v_2)_{\mathcal T_h}.
			\end{split}
		\end{align}
	\end{lemma}

	\subsubsection*{Step 5: Estimate for $\varepsilon_h^{\mathbb{G}}$}
	Before we estimate $\varepsilon_h^{\mathbb{G}}$, we give the following discrete Poincar\'{e} inequality from \cite[Proposition A.2]{MR3626531}.
	\begin{lemma}\label{lemma:discr_Poincare_ineq}
		We have
		\begin{align}\label{poin_in}
			\|\varepsilon_h^{\bm z}\|_{\mathcal T_h} \le C(\|\nabla \varepsilon_h^{\bm z}\|_{\mathcal T_h} + h^{-\frac 1 2} \|\varepsilon_h^{\bm z} - \varepsilon_h^{\widehat {\bm z}}\|_{\partial\mathcal T_h}).
		\end{align}
	\end{lemma}
	\begin{lemma}\label{lemma:step6_main_lemma}
		We have
		\begin{subequations}
			\begin{align}
				\hspace{3em}&\hspace{-3em}\|\varepsilon_h^{\mathbb G}\|_{\mathcal{T}_h}+h^{-\frac 1 2}\|{\bm P_M\varepsilon_h^{\bm z}-\varepsilon_h^{\widehat{\bm z}}}\|_{\partial \mathcal T_h}\nonumber\\
				&\lesssim h^{s_{\mathbb L}+1}\norm{\mathbb L}_{s_{\mathbb L},\Omega} + h^{s_{\bm y}}\norm{\bm y}_{s_{\bm y},\Omega}+ h^{s_{\mathbb G}}\norm{\mathbb G}_{s_{\mathbb G},\Omega} + h^{s_{\bm z}-1}\norm{\bm z}_{s_{\bm z},\Omega},\label{G_error}\\
				\|\varepsilon_h^{\bm z}\|_{\mathcal{T}_h} &\lesssim h^{s_{\mathbb L}+1}\norm{\mathbb L}_{s_{\mathbb L},\Omega} + h^{s_{\bm y}}\norm{\bm y}_{s_{\bm y},\Omega}+ h^{s_{\mathbb G}}\norm{\mathbb G}_{s_{\mathbb G},\Omega}+ h^{s_{\bm z}-1}\norm{\bm z}_{s_{\bm z},\Omega}.\label{z_error}
			\end{align}
		\end{subequations}
	\end{lemma}
	
	\begin{proof}
		First, we note the key inequality in  Lemma \ref{firstlemmastep1} is valid with $ (\mathbb{L},\bm y, \hat {\bm y}) $ in place of $ (\mathbb{G},\bm z,\hat{\bm z} ) $.  This gives
		\begin{align}\label{grad_z_h}
			\|\nabla \varepsilon_h^{\bm z}\|_{\mathcal{T}_h} +h^{-\frac{1}{2}}\|\varepsilon_h^{\bm z} -\varepsilon_h^{\widehat{\bm z}}\|_{\partial \mathcal{T}_h} {\lesssim} \|\varepsilon_h^{\mathbb{G}}\|_{\mathcal{T}_h} +h^{-\frac{1}{2}} \|\bm P_{ M} \varepsilon_h^{\bm z} -\varepsilon_h^{\widehat{\bm z}}\|_{\partial \mathcal{T}_h},
		\end{align}
		which we use below. Next, since $\varepsilon_h^{\widehat{\bm z}}=0$ on $\mathcal E_h^\partial$, the property of $\mathscr B$ in \eqref{pro_B2} gives
		\begin{align}
			\mathscr B(\varepsilon_h^{\mathbb{G}},\varepsilon_h^{\bm z},-\varepsilon_h^q,-\varepsilon_h^{\widehat q}, \varepsilon_h^{\widehat{\bm z}};\varepsilon_h^{\mathbb{G}},\varepsilon_h^{\bm z},-\varepsilon_h^q,-\varepsilon_h^{\widehat q}, \varepsilon_h^{\widehat{\bm z}})=\|\varepsilon_h^{\mathbb{G}}\|^2_{\mathcal T_h}+h^{-1}\|\bm P_{ M} \varepsilon_h^{\bm z}-\varepsilon_h^{\widehat{\bm z}}\|^2_{\partial\mathcal T_h}.
		\end{align}
		Next, we  take $(\mathbb T_2,\bm v_2,w_2,\widehat w_2, \bm \mu_2)=(\varepsilon_h^{\mathbb{G}},\varepsilon_h^{\bm z},-\varepsilon_h^q,-\varepsilon_h^{\widehat q}, \varepsilon_h^{\widehat{\bm z}})$ in \eqref{error_eq_Gzq} gives
		\begin{align*}
			\hspace{1em}&\hspace{-1em} \mathscr B(\varepsilon_h^{\mathbb{G}},\varepsilon_h^{\bm z},-\varepsilon_h^q,-\varepsilon_h^{\widehat q}, \varepsilon_h^{\widehat{\bm z}};\varepsilon_h^{\mathbb{G}},\varepsilon_h^{\bm z},-\varepsilon_h^q,-\varepsilon_h^{\widehat q}, \varepsilon_h^{\widehat{\bm z}}) \\
			& =  -\langle \delta^{\bm z}, \bm P_M \varepsilon_h^{\bm z}-\varepsilon_h^{\widehat{\bm z}}\rangle_{\partial\mathcal{T}_h} + \langle  \delta^{\mathbb G}\bm n, \varepsilon_h^{\bm z}-\varepsilon_h^{\widehat{\bm z}}\rangle_{\partial\mathcal{T}_h}+(\bm y-\bm y_h( u),\varepsilon_h^{\bm z})_{\mathcal T_h}.
		\end{align*}
		The estimate in \eqref{grad_z_h},  Lemmas \ref{lemma:step4_conv_rates} and \ref{lemma:discr_Poincare_ineq} and Young's inequality give the desired result.
	\end{proof}
	
	As a consequence of Lemma \ref{lemma:step6_main_lemma} and a simple application of the triangle inequality {we obtain the} optimal convergence rates for $\|\mathbb G- \mathbb G_h( u)\|_{\mathcal T_h}$ and $\| \bm z - \bm z_h( u)\|_{\mathcal T_h}$:
	\begin{lemma}\label{lemma:step7_conv_rates}
		Let $(\mathbb G, \bm z, q)$ and $(\mathbb G_h(u), \bm z_h(u), p_h(u))$ be the {solution} of \eqref{optimality_system} and \eqref{HDG_u_b}, respectively. We have
		\begin{subequations}
			\begin{align}
				\|\mathbb{G} -\mathbb{G}_h( u)\|_{\mathcal{T}_h}&\lesssim h^{s_{\mathbb L}+1}\norm{\mathbb L}_{s_{\mathbb L},\Omega}+ h^{s_{\bm y}}\norm{\bm y}_{s_{\bm y},\Omega} + h^{s_{\mathbb G}}\norm{\mathbb G}_{s_{\mathbb G},\Omega},\\
				\|\bm z-\bm z_h( u)\|_{\mathcal{T}_h}&\lesssim h^{s_{\mathbb L}+1}\norm{\mathbb L}_{s_{\mathbb L},\Omega}+ h^{s_{\bm y}}\norm{\bm y}_{s_{\bm y},\Omega} + h^{s_{\mathbb G}}\norm{\mathbb G}_{s_{\mathbb G},\Omega}+ h^{s_{\bm z}}\norm{\bm z}_{s_{\bm z},\Omega}.
			\end{align}
		\end{subequations}
	\end{lemma}

	\subsubsection*{Step 6: Estimates for $\|u-u_h\|_{\mathcal E_h^\partial}$ and $\norm {y-y_h}_{\mathcal T_h}$}
	
	Next, we bound the error between the solutions of the auxiliary problem and the HDG problem \eqref{HDG_full_discrete}.  We use these error bounds and the error bounds in  Lemmas \ref{lemma:step4_conv_rates} and \ref{lemma:step7_conv_rates} to obtain the main results.
	For the next step, we denote
	\begin{align*}
		\zeta_{\mathbb L}&=\mathbb L_h( u)-\mathbb L_h,\quad\zeta_{\bm y}=\bm y_h( u)-\bm y_h,\quad\zeta_p=p_h( u)- p_h, \quad \zeta_{\widehat p}=\widehat p_h( u)- \widehat p_h,\\
		\zeta_{\mathbb G}&=\mathbb G_h( u)-\mathbb G_h,\quad\zeta_{\bm z}=\bm z_h( u)-\bm z_h,\quad\zeta_q=q_h( u)- q_h, \quad \zeta_{\widehat q}=\widehat q_h( u)- \widehat q_h,
	\end{align*}
	and
	\begin{align*}
		\zeta_{\widehat {\bm y}}&=\widehat{\bm y}^o_h( u)-\widehat{\bm y}_h^o \;\; \textup{on} \;\; \varepsilon_h^{o} \;\;\textup{and} \;\;\zeta_{\widehat {\bm y}} =  P_M u \bm\tau -  u_h\bm \tau \;\; \textup{on} \;\; \mathcal E_h^{\partial} ,\\
		\zeta_{\widehat {\bm z}}&=\widehat{\bm z}^o_h( u)-\widehat{\bm z}_h^o \;\; \textup{on} \;\; \varepsilon_h^{o} \;\;\textup{and}  \;\;\zeta_{\widehat {\bm z}} = 0\;\; \textup{on} \;\; \mathcal E_h^{\partial}.
	\end{align*}
	Subtracting the auxiliary problem and the HDG problem gives the following error equations
	\begin{subequations}\label{eq_yh}
		\begin{align}
			\mathscr B(\zeta_{\mathbb L},\zeta_{\bm y},\zeta_p,\zeta_{\widehat p},\zeta_{\widehat{\bm y}};\mathbb T_1,\bm v_1, w_1,\widehat w_1,\bm \mu_1)&=\langle  (P_M  u - u_h)\bm\tau, h^{-1} \bm v_1+\mathbb T_1\bm n \rangle_{\mathcal E_h^\partial},\label{eq_yh1}\\
			\mathscr B(\zeta_{\mathbb G},\zeta_{\bm z},-\zeta_q,-\zeta_{\widehat q},\zeta_{\widehat{\bm z}};\mathbb T_2,\bm v_2, w_2,\widehat w_2,\bm \mu_2)&=(\zeta_{\bm y},\bm v_2)_{\mathcal{T}_h}\label{eq_yh2}
		\end{align}
	\end{subequations}
	for all $(\mathbb T_1,\bm v_1,w_1,\widehat{w}_1,\bm \mu_1;\mathbb T_2,\bm v_2,w_2,\widehat{w}_2,\bm \mu_2) \in [\mathbb K_h\times\bm{V}_h\times W_h^0\times Q_h\times  \bm M_h(o)]^2$.
	
	\begin{lemma}
		We have
		\begin{align}\label{eq_uuh_yhuyh}
			\begin{split}
				\gamma\| u-  u_h\|^2_{\mathcal E_h^\partial}+\|\zeta_{\bm y}\|^2_{\mathcal T_h} &= \langle \gamma  u\bm{\tau}-\mathbb G_h( u) \bm n+h^{-1}  \bm P_M\bm z_h( u), (u-  u_h)\bm{\tau}\rangle_{\mathcal E_h^\partial} \\
				& \quad- \langle \gamma  u_h \bm \tau- \mathbb G_h\bm n +h^{-1} \bm P_M\bm z_h,  (u-  u_h)\bm{\tau} \rangle_{\mathcal E_h^\partial}.
			\end{split}
		\end{align}
	\end{lemma}
	\begin{proof}
		First, we have
		\begin{align*}
			&\quad\langle \gamma  u\bm \tau-\mathbb G_h( u) \bm n+h^{-1}  \bm P_M\bm z_h( u), (u- u_h)\bm{\tau}\rangle_{\mathcal E_h^\partial}\\
			&- \langle \gamma  u_h \bm \tau- \mathbb G_h\bm n +h^{-1} \bm P_M\bm z_h, ( u- u_h)\bm{\tau}\rangle_{\mathcal E_h^\partial}\\
			&= \gamma\norm{ u- u_h}_{\mathcal E_h^\partial}^2 +\langle -\zeta_{\mathbb G} \bm n+ h^{-1}  \bm P_M \zeta_{\bm z}, (u-  u_h)\bm{\tau}\rangle_{\mathcal E_h^\partial}.
		\end{align*}
		Next, Lemma \ref{identical_equal} gives
		\begin{align*}
			\mathscr B(\zeta_{\mathbb L},\zeta_{\bm y},\zeta_{p},\zeta_{\widehat p},  \zeta_{\widehat{\bm y}}; - \zeta_{\mathbb G},\zeta_{\bm z},\zeta_{q}, \zeta_{\widehat q},\zeta_{\widehat{\bm z}}) =\mathscr B ( \zeta_{\mathbb G},\zeta_{\bm z},-\zeta_{q}, -\zeta_{\widehat q},\zeta_{\widehat{\bm z}}; -\zeta_{\mathbb L},\zeta_{\bm y},\zeta_{p},\zeta_{\widehat p}, \zeta_{\widehat{\bm y}}).
		\end{align*}
		On the other hand, from (\ref{eq_yh1}) and (\ref{eq_yh2}) we have
		\begin{align*}
			\hspace{1em}&\hspace{-1em}  \mathscr B(\zeta_{\mathbb L},\zeta_{\bm y},\zeta_{p},\zeta_{\widehat p},  \zeta_{\widehat{\bm y}}; - \zeta_{\mathbb G},\zeta_{\bm z},\zeta_{q}, \zeta_{\widehat q},\zeta_{\widehat{\bm z}}) -\mathscr B ( \zeta_{\mathbb G},\zeta_{\bm z},-\zeta_{q}, -\zeta_{\widehat q},\zeta_{\widehat{\bm z}}; -\zeta_{\mathbb L},\zeta_{\bm y},\zeta_{p},\zeta_{\widehat p}, \zeta_{\widehat{\bm y}})\\
			&=  -(\zeta_{ \bm y},\zeta_{ \bm y})_{\mathcal{T}_h} + \langle  P_M ( u- u_h) \bm \tau, -\zeta_{\mathbb G} \bm{n} + h^{-1} \zeta_{ \bm z}  \rangle_{{\mathcal E_h^{\partial}}}\\
			&= -(\zeta_{\bm y},\zeta_{ \bm y})_{\mathcal{T}_h} + \langle   (u- u_h) \bm \tau, -\zeta_{\mathbb G} \bm{n} + h^{-1} \bm P_M \zeta_{ \bm z}  \rangle_{{\mathcal E_h^{\partial}}}.
		\end{align*}
		
		Comparing the above two equalities gives
		\begin{align*}
			(\zeta_{\bm y},\zeta_{ \bm y})_{\mathcal{T}_h} =  \langle   (u- u_h) \bm \tau, -\zeta_{\mathbb G} \bm{n} + h^{-1} \bm P_M \zeta_{ \bm z}  \rangle_{{\mathcal E_h^{\partial}}}.
		\end{align*}
	\end{proof}
	
	\begin{theorem}
		Let $(\bm y, u)$ and $(\bm y_h, u_h)$ be the {solutions} of \eqref{optimality_system} and \eqref{HDG_full_discrete}, respectively. We have
		\begin{subequations}
			\begin{align}
				\norm{ u-  u_h}_{\mathcal E_h^\partial}&\lesssim h^{s_{\mathbb L}+\frac 1 2}\norm{\mathbb L}_{s_{\mathbb L},\Omega}  +  h^{s_{\bm y}-\frac 1 2}\norm{\bm y}_{s_{\bm y},\Omega} + h^{s_{\mathbb G}-\frac 1 2}\norm{\mathbb G}_{s_{\mathbb G},\Omega}{+h^{s_{\bm z}-\frac 3 2}\norm{\bm z}_{s_{\bm z},\Omega}},\\
				\norm {\bm y-\bm y_h}_{\mathcal T_h} &\lesssim h^{s_{\mathbb L}+\frac 1 2}\norm{\mathbb L}_{s_{\mathbb L},\Omega}  +  h^{s_{\bm y}-\frac 1 2}\norm{\bm y}_{s_{\bm y},\Omega} + h^{s_{\mathbb G}-\frac 1 2}\norm{\mathbb G}_{s_{\mathbb G},\Omega}{+h^{s_{\bm z}-\frac 3 2}\norm{\bm z}_{s_{\bm z},\Omega}}.
			\end{align}
		\end{subequations}
	\end{theorem}
	
	\begin{proof}
		Since $\gamma u\bm \tau-\mathbb G\bm n = 0$ on $\mathcal E_h^\partial$ and $\gamma  u_h \bm \tau- \mathbb G_h \bm n + h^{-1} \bm P_M \bm z_h = 0$ on $\mathcal E_h^\partial$ we have
		%	
		%	
		%	 $\bm u_h\in \bm M_h^0(\partial)\subset \bm L_0^2(\Gamma)$, we have $\langle \gamma \bm u-\mathbb G\bm n -q\bm n, \bm u-\bm u_h\rangle_{\mathcal E_h^\partial}=0$.
		%	Moreover
		%	\begin{align*}
		%	\hspace{3em}&\hspace{-3em} \langle	\gamma \bm u_h - \mathbb G_h \bm n -q_h \bm n + h^{-1} \bm P_M \bm z_h, \bm u-\bm u_h\rangle_{\mathcal E_h^\partial} \\
		%	& = 	\langle	\gamma \bm u_h - \mathbb G_h \bm n -q_h \bm n + h^{-1} \bm P_M \bm z_h, \bm P_M\bm u-\bm u_h\rangle_{\mathcal E_h^\partial}\\
		%	&=0.
		%	\end{align*}
		%Here we used $\gamma \bm u_h - \mathbb G_h \bm n -q_h \bm n + h^{-1} \bm P_M \bm z_h \in \bm M_h$ and $\bm P_M \bm u \in \bm M_h^0(\partial)$.
		%
		%This implies
		%
		\begin{align*}
			\gamma\norm{ u- u_h}_{\mathcal E_h^\partial}^2  + \norm {\zeta_{\bm y}}_{\mathcal T_h}^2 &=  \langle \gamma  u \bm \tau-\mathbb G_h( u) \bm n+h^{-1}  \bm P_M\bm z_h( u), (u-  u_h)\bm \tau\rangle_{\mathcal E_h^\partial}\\
			&=\langle (\mathbb G -\mathbb G_h( u))\bm n + h^{-1} \bm P_M \bm z_h( u),  (u-  u_h)\bm \tau \rangle_{\mathcal E_h^\partial}.
		\end{align*}
		Next, since $\widehat {\bm z}_h( u) = \bm z = \bm 0$ on $\mathcal E_h^{\partial}$ we have
		\begin{align*}
			\|\bm P_M \bm z_h(u)\|_{\mathcal E_h^\partial} &={\|\bm P_M \bm z_h(u)  - \bm P_M \Pi^{\textup{RT}} \bm z +\bm P_M \Pi^{\textup{RT}} \bm z-\bm P_M \bm z+ \bm P_M \bm z- \widehat {\bm z}_h( u) \|_{\mathcal E_h^\partial}}\\
			& \leq  \|\bm P_M \varepsilon_h^{\bm z} -\varepsilon_h^{\widehat {\bm z}}\|_{\partial\mathcal T_h}+\|\Pi^{\textup{RT}} \bm z-\bm z\|_{\mathcal E_h^{\partial}}.
		\end{align*}
		This together with Lemma \ref{lemma:step7_conv_rates}  gives
		\begin{align*}
			\norm{ u- u_h}_{\mathcal E_h^\partial}  + \|\zeta_{\bm y}\|_{\mathcal T_h}&\lesssim h^{-\frac 1 2}\norm {\varepsilon_h^{\mathbb G}}_{\mathcal T_h} +h^{s_{\mathbb G}-\frac 12 } \norm{\mathbb G}_{s_{\mathbb G},\Omega}\\ &\quad +h^{-1}\|\bm P_M \varepsilon_h^{\bm z} -\varepsilon_h^{\widehat {\bm z}}\|_{\partial\mathcal T_h}{+h^{-1}\|\Pi^{\textup{RT}} \bm z-\bm z\|_{\mathcal E_h^{\partial}}}.
		\end{align*}
		By Lemma \ref{lemma:step6_main_lemma} and properties of the $ L^2 $ projection, we have
		\begin{align*}
			\norm{ u- u_h}_{\mathcal E_h^\partial}  + \|\zeta_{\bm y}\|_{\mathcal T_h}\lesssim h^{s_{\mathbb L}+\frac 1 2}\norm{\mathbb L}_{s_{\mathbb L},\Omega}  +  h^{s_{\bm y}-\frac 1 2}\norm{\bm y}_{s_{\bm y},\Omega}
			+ h^{s_{\mathbb G}-\frac 1 2}\norm{\mathbb G}_{s_{\mathbb G},\Omega}+  {h^{s_{\bm z}-\frac 3 2}\norm{\bm z}_{s_{\bm z},\Omega}}.
		\end{align*}
		Then, by the triangle inequality and Lemma \ref{lemma:step4_conv_rates} we obtain
		\begin{align*}
			\|\bm y -\bm y_h\|_{\mathcal T_h}  \lesssim h^{s_{\mathbb L}+\frac 1 2}\norm{\mathbb L}_{s_{\mathbb L},\Omega}  +  h^{s_{\bm y}-\frac 1 2}\norm{\bm y}_{s_{\bm y},\Omega}+ h^{s_{\mathbb G}-\frac 1 2}\norm{\mathbb G}_{s_{\mathbb G},\Omega}{+h^{s_{\bm z}-\frac 3 2}\norm{\bm z}_{s_{\bm z},\Omega}}.
		\end{align*}
	\end{proof}

	\subsubsection*{Step 7: Estimates for $\|\mathbb G-\mathbb G_h\|_{\mathcal T_h}$ and $\|z-z_h\|_{\mathcal T_h}$}
	\begin{lemma}
		We have
		\begin{subequations}
			\begin{align}
				\norm {\zeta_{\mathbb G}}_{\mathcal T_h}  &\lesssim h^{s_{\mathbb L}+\frac 1 2}\norm{\mathbb L}_{s_{\mathbb L},\Omega} + h^{s_{ p}+\frac 1 2}\norm{p}_{s^{p},\Omega} +  h^{s_{\bm y}-\frac 1 2}\norm{\bm y}_{s_{\bm y},\Omega}\nonumber\\
				&\quad + h^{s_{\mathbb G}-\frac 1 2}\norm{\mathbb G}_{s_{\mathbb G},\Omega} + h^{s_{ q}-\frac 1 2}\norm{q}_{s^{q},\Omega} + h^{s_{\bm z}-\frac 3 2}\norm{\bm z}_{s_{\bm z},\Omega},\\
				\|\zeta_{\bm z}\|_{\mathcal T_h} &\lesssim h^{s_{\mathbb L}+\frac 1 2}\norm{\mathbb L}_{s_{\mathbb L},\Omega} + h^{s_{ p}+\frac 1 2}\norm{p}_{s^{p},\Omega} +  h^{s_{\bm y}-\frac 1 2}\norm{\bm y}_{s_{\bm y},\Omega}\nonumber\\
				&\quad + h^{s_{\mathbb G}-\frac 1 2}\norm{\mathbb G}_{s_{\mathbb G},\Omega} + h^{s_{ q}-\frac 1 2}\norm{q}_{s^{q},\Omega} + h^{s_{\bm z}-\frac 3 2}\norm{\bm z}_{s_{\bm z},\Omega}.
			\end{align}
		\end{subequations}
	\end{lemma}
	\begin{proof}
		By Lemma \ref{property_B}, the error equation  \eqref{eq_yh2}, and since $\zeta_{ \widehat {\bm z}} = 0$ on $\mathcal E_h^{\partial}$, we have
		\begin{align*}
			\hspace{2em}&\hspace{-2em}  {(\zeta_{\mathbb G},\zeta_{\mathbb G})_{{\mathcal{T}_h}}+h^{-1}  \|\bm P_M \zeta_{\bm z} -\zeta_{\widehat {\bm z}}\|_{\partial\mathcal T_h}^2} \\
			&={\mathscr B ( \zeta_{\mathbb G},\zeta_{\bm z},-\zeta_{q},-\zeta_{\widehat{q}}, \zeta_{\widehat{\bm z}}; \zeta_{\mathbb G},\zeta_{\bm z},-\zeta_{q},-\zeta_{\widehat{q}}, \zeta_{\widehat{\bm z}})} \\						
			&=(\zeta_{\bm y},\zeta_{\bm z})_{\mathcal T_h}\\
			&\le \norm{\zeta_{\bm y}}_{\mathcal T_h} \norm{\zeta_{\bm z}}_{\mathcal T_h}\\
			&\lesssim \norm{\zeta_{\bm y}}_{\mathcal T_h} (\|\nabla \zeta_{\bm z}\|_{\mathcal T_h} + h^{-\frac 1 2} \| \zeta_{\bm z} - \zeta_{\widehat {\bm z}}\|_{\partial\mathcal T_h}) \\
			&\lesssim \norm{\zeta_{\bm y}}_{\mathcal T_h} (\|\zeta_{\mathbb G}\|_{\mathcal T_h} + h^{-\frac 1 2} \|\bm P_M \zeta_{\bm z} - \zeta_{\widehat {\bm z}}\|_{\partial\mathcal T_h}),
		\end{align*}
		where we used the discrete Poincar\'{e} inequality in Lemma \ref{lemma:discr_Poincare_ineq} and also \eqref{grad_y_h}.  This implies
		% we get
		\begin{align*}
			\norm {\zeta_{\mathbb G}}_{\mathcal T_h} +h^{-\frac1 2}\|\bm P_M\zeta_{\bm z}-\zeta_{\widehat {\bm z}}\|_{\partial\mathcal T_h} &{\lesssim  \norm{\zeta_{\bm y}}_{\mathcal T_h}}\\
			&\lesssim h^{s_{\mathbb L}+\frac 1 2}\norm{\mathbb L}_{s_{\mathbb L},\Omega}  +  h^{s_{\bm y}-\frac 1 2}\norm{\bm y}_{s_{\bm y},\Omega}\\
			&+ h^{s_{\mathbb G}-\frac 1 2}\norm{\mathbb G}_{s_{\mathbb G},\Omega}{+h^{s_{\bm z}-\frac 3 2}\norm{\bm z}_{s_{\bm z},\Omega}}.
		\end{align*}
		
		The discrete Poincar\'{e} inequality in Lemma \ref{lemma:discr_Poincare_ineq} also gives
		\begin{align*}
			\|\zeta_{\bm z}\|_{\mathcal T_h} & \lesssim \|\nabla \zeta_{\bm z}\|_{\mathcal T_h} + h^{-\frac 1 2} \|\zeta_{\bm z} - \zeta_{\widehat {\bm{z}}}\|_{\partial\mathcal T_h}\\
			&\lesssim \|\zeta_{\mathbb G}\|_{\mathcal T_h} + h^{-\frac 1 2} \|\bm P_M \zeta_{\bm z} - \zeta_{\widehat {\bm z}}\|_{\partial\mathcal T_h}\\
			&\lesssim h^{s_{\mathbb L}+\frac 1 2}\norm{\mathbb L}_{s_{\mathbb L},\Omega}  +  h^{s_{\bm y}-\frac 1 2}\norm{\bm y}_{s_{\bm y},\Omega}+ h^{s_{\mathbb G}-\frac 1 2}\norm{\mathbb G}_{s_{\mathbb G},\Omega}{+h^{s_{\bm z}-\frac 3 2}\norm{\bm z}_{s_{\bm z},\Omega}}.
		\end{align*}
	\end{proof}
	
	The above lemma along with the triangle inequality and Lemmas \ref{lemma:step4_conv_rates} and \ref{lemma:step7_conv_rates} gives the next part of the main result:
	\begin{theorem}
		Let $(\mathbb G, \bm z)$ and $(\mathbb G_h, \bm z_h)$ be the {solutions} of \eqref{optimality_system} and \eqref{HDG_full_discrete}, respectively. We have
		\begin{subequations}
			\begin{align}
				\norm {\mathbb G - \mathbb G_h}_{\mathcal T_h}   & \lesssim h^{s_{\mathbb L}+\frac 1 2}\norm{\mathbb L}_{s_{\mathbb L},\Omega}  +  h^{s_{\bm y}-\frac 1 2}\norm{\bm y}_{s_{\bm y},\Omega}+ h^{s_{\mathbb G}-\frac 1 2}\norm{\mathbb G}_{s_{\mathbb G},\Omega}{+h^{s_{\bm z}-\frac 3 2}\norm{\bm z}_{s_{\bm z},\Omega}},\\
				\norm {\bm z - \bm z_h}_{\mathcal T_h}  &\lesssim h^{s_{\mathbb L}+\frac 1 2}\norm{\mathbb L}_{s_{\mathbb L},\Omega}  +  h^{s_{\bm y}-\frac 1 2}\norm{\bm y}_{s_{\bm y},\Omega}+ h^{s_{\mathbb G}-\frac 1 2}\norm{\mathbb G}_{s_{\mathbb G},\Omega}{+h^{s_{\bm z}-\frac 3 2}\norm{\bm z}_{s_{\bm z},\Omega}}.
			\end{align}
		\end{subequations}
	\end{theorem}
	
	\subsubsection*{Step 8: Estimate for $\|\mathbb L - \mathbb L_h\|_{\mathcal T_h}$}
	
	\begin{lemma}
		If  $k\geq 1$ holds, then
		\begin{align}
			\norm {\zeta_{\mathbb L}}_{\mathcal T_h}\lesssim h^{s_{\mathbb L}}\norm{\mathbb L}_{s_{\mathbb L},\Omega} + h^{s_{\bm y}-1}\norm{\bm y}_{s_{\bm y},\Omega} +h^{s_{\mathbb G}-1}\norm{\mathbb G}_{s_{\mathbb G},\Omega}{+h^{s_{\bm z}-2}\norm{\bm z}_{s_{\bm z},\Omega}}.\label{eL}
		\end{align}
	\end{lemma}
	\begin{proof}
		By Lemma \ref{property_B} and the error equation \eqref{eq_yh1}, we have
		\begin{align*}
			\hspace{2em}&\hspace{-2em}  {(\zeta_{\mathbb L}, \zeta_{\mathbb L})_{{\mathcal{T}_h}}+\langle (h^{-1} (\bm P_M \zeta_{\bm y}-\zeta_{\widehat {\bm y}}) , \zeta_{\bm y}-\zeta_{\widehat {\bm y}} \rangle_{\partial{{\mathcal{T}_h}}\backslash\mathcal E_h^\partial}+ \langle h^{-1}\bm P_M \zeta_{\bm y}, \bm P_M\zeta_{\bm y} \rangle_{\mathcal E_h^\partial}}\\
			&=\mathscr B ( \zeta_{\mathbb L},\zeta_{\bm y},\zeta_{p}, \zeta_{\widehat{\bm y}}; \zeta_{\mathbb L},\zeta_{\bm y},\zeta_{p}, \zeta_{\widehat{\bm y}}) \\
			&= \langle (P_M  u -  u_h)\bm{\tau}, \zeta_{\mathbb L}\cdot \bm{n} + h^{-1} \zeta_{\bm y} \rangle_{{\mathcal E_h^{\partial}}}\\
			& =\langle   (u-  u_h)\bm \tau, \zeta_{\mathbb L}\cdot \bm{n} + h^{-1} \bm P_M \zeta_{\bm y} \rangle_{{\mathcal E_h^{\partial}}}\\
			&	\lesssim \norm { u- u_h}_{\mathcal E_h^{\partial}} (\norm {\zeta_{\mathbb L}}_{\mathcal E_h^{\partial}} + h^{-1} \norm {\bm P_M \zeta_{\bm y}}_{\mathcal E_h^{\partial}})\\
			&	\lesssim h^{-\frac 1 2}\norm { u- u_h}_{\mathcal E_h^{\partial}} (\norm {\zeta_{\mathbb L}}_{\mathcal T_h} + h^{-\frac 1 2} \norm {\bm P_M \zeta_{\bm y}}_{\mathcal E_h^{\partial}}),
		\end{align*}
		which gives
		\begin{align*}
			\norm {\zeta_{\mathbb L}}_{\mathcal T_h} {\lesssim h^{-\frac 1 2}\norm { u- u_h}_{\mathcal E_h^{\partial}}} \lesssim h^{s_{\mathbb L}}\norm{\mathbb L}_{s_{\mathbb L},\Omega} + h^{s_{\bm y}-1}\norm{\bm y}_{s_{\bm y},\Omega} +h^{s_{\mathbb G}-1}\norm{\mathbb G}_{s_{\mathbb G},\Omega}{+h^{s_{\bm z}-2}\norm{\bm z}_{s_{\bm z},\Omega}}.
		\end{align*}
	\end{proof}

	The above lemma along with the triangle inequality and Lemmas \ref{lemma:step4_conv_rates} and \ref{lemma:step7_conv_rates}  {completes} the proof of the main result:
	\begin{theorem}
		Let $\mathbb L$ and $\mathbb L_h$ be the  {solutions} of \eqref{optimality_system} and \eqref{HDG_full_discrete}, respectively. If  $k\geq 1$ holds, then
		\begin{align*}
			\norm {\mathbb L - \mathbb L_h}_{\mathcal T_h} \lesssim h^{s_{\mathbb L}}\norm{\mathbb L}_{s_{\mathbb L},\Omega} + h^{s_{\bm y}-1}\norm{\bm y}_{s_{\bm y},\Omega} +h^{s_{\mathbb G}-1}\norm{\mathbb G}_{s_{\mathbb G},\Omega}{+h^{s_{\bm z}-2}\norm{\bm z}_{s_{\bm z},\Omega}}.
		\end{align*}
	\end{theorem}
	
	\section{Numerical experiments}
	\label{Numerical_experiments}
	
	In this section, we present  {some} numerical experiments to illustrate our theoretical results (see Theorem \ref{main_res}).  We use  uniform triangular meshes and  define
	\begin{align*}
		\textup{div}(\bm y_h)=\max_{K\in\mathcal T_h} \dfrac{1}{|K|}\int_K|\nabla\cdot\bm y_h| ~{\rm d}\bm x.		
	\end{align*}

	\begin{example}\label{example1}
		We begin with an example which has an  {analytical} solution. The domain is the unit square $\Omega = (0,1)^2$ and the data is chosen as
		\begin{gather*}
			y_1= -2\pi^2\sin^2(\pi x_1)\cos(\pi x_2) -2\pi^2\sin(\pi x_1)\sin(2\pi x_2), \\
			y_2 = 2\pi^2\cos(\pi x_1) \sin^2(\pi x_2) +2\pi^2\sin(\pi x_2) \sin(2\pi x_1),\\
			\ \ z_1= \pi \sin^2(\pi x_1) \sin(2\pi x_2),
			\ \  z_2= -\pi\sin^2(\pi x_2) \sin(2\pi x_1),\\
			p=  10^n\cos(\pi x_1),\ \ q= 10^n \cos(\pi x_1),\ \ \gamma = 1.
		\end{gather*}
		Here  $n$ is a parameter.
	\end{example}

    To make a comparison, we first   {solve} the optimality system  \eqref{optimality_system} by using the HDG method proposed in \cite{GongStokes_Tangential1},  with $n=2, 4, 6$ and $k=0$. The errors for all variables are shown in Tables \ref{Table5} and \ref{Table6}. Although the convergence rates are optimal  {and consistent with} the error analysis in \cite{GongStokes_Tangential1} for $n=4, 6$, the magnitude of the errors strongly depend on the pressures. This shows that the algorithm proposed and analyzed in \cite{GongStokes_Tangential1} is not pressure-robust.

    	\begin{table}[H]\small
    	\caption{Example \ref{example1}:   Lack of pressure-robustness:   Errors and observed convergence orders for the control $u$, pressure $p$, state $\bm y$,  and its flux $\mathbb L$ by using the HDG method in \cite{GongStokes_Tangential1}.}
    	\label{Table5}
    	\centering
    	\begin{tabular}{c|c|c|c|c|c|c|c|c|c|c|c}
    		\Xhline{1pt}
    		
    		\multirow{2}{*}{$k$} &
    		\multirow{2}{*}{$n$} &
    		\multirow{2}{*}{$\frac{\sqrt{2}}{h}$} &
    		\multirow{2}{*}{$\textup{div}(\bm y_h)$}&
    		
    		\multicolumn{2}{c|}{$\|\bm y-\bm y_h\|_{L^2(\Omega)}$} &

    		\multicolumn{2}{c|}{$\|\mathbb L-\mathbb L_h\|_{L^2(\Omega)}$}&
    		
    		\multicolumn{2}{c|}{$\|p-p_h\|_{L^2(\Omega)}$}&
    		
    		\multicolumn{2}{c}{$\|u-u_h\|_{L^2(\Gamma)}$}
    		\\
    		
    		\cline{5-12}
    		& 	& 	& &Error &Rate  &Error &Rate &Error &Rate  &Error &Rate  \\
    		
    		\Xhline{1pt}
    		
    		&		&	4	&	7.52E+01	&	1.18E+01	&		&	5.58E+01	&		&	3.63E+01	&		&1.77E+01	&		\\
    		&		&	8	&	3.36E+01	&	3.58E+00	&	1.72 	&	2.80E+01	&	0.99 	&2.45E+01	&	0.56 	&	6.33E+00	&	1.48	\\
    		0	&	$2$	&	16	&	1.59E+01	&	1.51E+00	&	1.24 	&	1.39E+01	&	0.99 	&	1.43E+01	&	0.78 	&	1.89E+00	&	1.74	\\
    		&		&	32	&	7.79E+00	&	1.03E+00	&	0.54	&	8.92E+00	&	0.64 	&	1.14E+01	&	0.32	&	5.65E-01	&	1.74 	\\
    		&		&	64	&	3.87E+00	&	9.31E-01	&	0.15 	&	7.49E+00	&	0.25 	&	1.10E+01	&	0.04 	&	2.69E-01	&	1.06 	\\
    		\Xhline{1pt}
    		&		&	4	&		2.74E+03	&	6.17E+02	&		&	4.49E+03	&		&	3.43E+03	&		&	1.74E+03	&		\\
    		&		&	8	&1.17E+03	&	1.69E+02	&	1.86 	&	2.21E+03	&	1.09 	&	2.19E+03	&	0.64 	&	6.24E+02	&	1.48 	\\
    		0	&	$4$	&	16	&	4.78E+02	&	4.56E+01	&	1.89 	&	8.30E+02	&	1.34 	&	9.17E+02	&	1.25 	&	1.83E+02	&	1.76 	\\
    		&		&	32	&	2.14E+02	&	1.19E+01	&	1.93 	&	3.33E+02	&	1.31	&	3.14E+02	&	1.54 	&	4.95E+01	&	1.89	\\
    		&		&	64	&		1.03E+02		&	3.15E+00	&	1.91	&	1.46E+02	&	1.18 	&	1.00E+02	&	1.65 	&	1.28E+01	&	1.94 	\\
    		\Xhline{1pt}
    		&		&	4	&	2.74E+05	&	6.17E+04	&		&	4.49E+05	&		&	3.43E+05	&		&	1.74E+05	&		\\
    	&		&	8	&	1.17E+05	&	1.69E+04	&	1.86 	&	2.21E+05	&	1.09 	&	2.19E+05	&	0.64 	&	6.24E+04	&	1.48	\\
    	0	&	$6$	&	16	&	4.78E+04	&	4.56E+03	&	1.89 	&	8.30E+04	&	1.34 	&	9.17E+04	&	1.25 	&	1.83E+04	&	1.76 	\\
    	&		&	32	&	2.14E+04	&	1.19E+03	&	1.93 	&	3.33E+04	&	1.31	&	3.14E+04	&	1.54 	&	4.95E+03	&	1.89	\\
    	&		&	64	&	1.03E+04	&	3.15E+02	&	1.91	&	1.46E+04	&	1.18 	&	1.00E+04	&	1.65 	&	1.28E+03	&	1.94 	\\
    		    		\Xhline{1pt}
    	\end{tabular}
    \end{table}

    \begin{table}[H]\small
    	\caption{Example \ref{example1}: Lack of pressure-robustness: Errors and observed convergence orders for the dual pressure $q$,  dual state $\bm z$,  and its flux $\mathbb G$ by using the HDG method in \cite{GongStokes_Tangential1}.}
    	\label{Table6}
    	\centering
    	\begin{tabular}{c|c|c|c|c|c|c|c|c|c}
    		\Xhline{1pt}
    		
    		\multirow{2}{*}{$k$} &
    		\multirow{2}{*}{$n$} &
    		\multirow{2}{*}{$\frac{\sqrt{2}}{h}$} &
    		\multirow{2}{*}{$\textup{div}(\bm z_h)$}&
    		
    		\multicolumn{2}{c|}{$\|\bm z-\bm z_h\|_{L^2(\Omega)}$} &

    		\multicolumn{2}{c|}{$\|\mathbb G-\mathbb G_h\|_{L^2(\Omega)}$}&
    		
    		\multicolumn{2}{c}{$\|q-q_h\|_{L^2(\Omega)}$}

    		\\
    		
    		\cline{5-10}
    		& & 	& &Error &Rate  &Error &Rate &Error &Rate    \\

    		\Xhline{1pt}
    		
    		&		&	4	&	1.19E+01	&	3.21E+00	&		&	1.35E+01	&		&	9.28E+00	&		\\
    		&		&	8	&	5.62E+00	&	8.98E-01	&	1.83 	&	7.80E+00	&	0.79	&	3.31E+00	&	1.48 	\\
    		0	&	$2$	&	16	&	2.76E+00&	2.30E-01	&	1.96 	&	4.06E+00	&	0.94 	&1.00E+00	&	1.71 	\\
    		&		&	32	&	1.38E+00	&	5.74E-02	&	2.00	&	2.05E+00	&	0.98 	&	3.16E-01	&	1.67 	\\
    		&		&	64	&	6.89E-01	&	1.96E-02	&	1.54 	&	1.03E+00	&	0.98	&	1.20E-01	&	1.39	\\
    		\Xhline{1pt}
    		
    		&		&	4	&	1.12E+03	&	3.08E+02	&		&	1.31E+03	&		&	9.15E+02	&		\\
    		&		&	8	&	5.11E+02	&	8.66E+01	&	1.83	&	7.59E+02	&	0.79 	&	3.21E+02	&	1.50 	\\
    		0	&	$4$	&	16	&2.51E+02	&	2.25E+01	&	1.94 	&	3.95E+02	&	0.94 	&	9.27E+01	&	1.79 	\\
    		&		&	32	&	1.25E+02		&	5.68E+00	&	1.98 	&	1.99E+02	&	0.98	&	2.45E+01	&	1.91 	\\
    		&		&	64	&		6.26E+01	&	1.42E+00	&	1.99 	&	1.00E+02	&	0.99 	&	6.25E+00	&	1.97 	\\
    		\Xhline{1pt}
    		
    	    		&		&	4	& 1.12E+05	&	3.08E+04	&		&	1.31E+05	&		&	9.15E+04	&		\\
    	&		&	8	&	5.11E+04	&	8.66E+03	&	1.83	&	7.59E+04	&	0.79 	&	3.21E+04	&	1.50 	\\
    	0	&	$6$	&	16	&	2.51E+04	&	2.25E+03	&	1.94 	&	3.95E+04	&	0.94 	&	9.27E+03	&	1.79 	\\
    	&		&	32	&	1.25E+04	&	5.68E+02	&	1.98 	&	1.99E+04	&	0.98	&	2.45E+03	&	1.91 	\\
    	&		&	64	&	6.26E+03	&	1.42E+02	&	1.99 	&	1.00E+04	&	0.99 	&	6.25E+02	&	1.97 	\\
    		\Xhline{1pt}
    	\end{tabular}
    \end{table}

	Now we use the new  HDG method (see the formulation \eqref{HDG_discrete2}) to test the same problem. The errors for all variables are shown in Tables \ref{Table1} and \ref{Table2}. We see that the error  {magnitudes} of the state $\bm y$, dual state $\bm z$ and control $u$ are \emph{independent} of the pressure $p$ and the dual pressure $q$. We also notice that the convergence rates are higher than predicted by our error analysis; a similar phenomena has been observed for other numerical methods for Dirichlet boundary control problems involving elliptic equations  \cite{HuShenSinglerZhangZheng_HDG_Dirichlet_control1,MR3992054,HuMateosSinglerZhangZhang2} and Stokes  equations \cite{GongStokes_Tangential1,GongStokes_energy1}.  To the best of our knowledge, only one work explained the above phenomena: May, Rannacher, and Vexler in \cite{MR3070527} used a duality argument to obtain improved convergence rates for the state and dual state with the standard finite element method.  It is not clear how to apply this technique to the HDG methods.
		{

		\begin{table}[H]\small
		\caption{Example \ref{example1}:  Pressure-robustness:  Errors and observed convergence orders for the control $u$, pressure $p$, state $\bm y$,  and its flux $\mathbb L$ by using the new HDG formulation \eqref{HDG_discrete2}.}
			\label{Table1}
			\centering
			\begin{tabular}{c|c|c|c|c|c|c|c|c|c|c|c}
				\Xhline{1pt}
				
				\multirow{2}{*}{$k$} &
				\multirow{2}{*}{$n$} &
				\multirow{2}{*}{$\frac{\sqrt{2}}{h}$} &
				\multirow{2}{*}{$\textup{div}(\bm y_h)$}&
				
				\multicolumn{2}{c|}{$\|\bm y-\bm y_h\|_{L^2(\Omega)}$} &

				\multicolumn{2}{c|}{$\|\mathbb L-\mathbb L_h\|_{L^2(\Omega)}$}&
				
				\multicolumn{2}{c|}{$\|p-p_h\|_{L^2(\Omega)}$}&
				
				\multicolumn{2}{c}{$\|u-u_h\|_{L^2(\Gamma)}$}
				\\
				
				\cline{5-12}
				& 	& 	& &Error &Rate  &Error &Rate &Error &Rate  &Error &Rate  \\
				
				\Xhline{1pt}
				
				&		&	4	&	8.88E-16	&	8.76E+00	&		&	5.33E+01	&		&	1.47E+01	&		&	6.30E+00	&		\\
				&		&	8	&	6.66E-16	&	2.20E+00	&	2.00 	&	2.79E+01	&	0.93 	&	7.20E+00	&	1.03 	&	3.15E+00	&	1.00 	\\
				0	&	$2$	&	16	&	3.33E-16	&	5.41E-01	&	2.02 	&	1.41E+01	&	0.99 	&	3.64E+00	&	0.99 	&	1.64E+00	&	0.94 	\\
				&		&	32	&	3.05E-16	&	1.34E-01	&	2.01 	&	7.05E+00	&	1.00 	&	1.80E+00	&	1.02 	&	8.06E-01	&	1.03 	\\
				&		&	64	&	1.94E-16	&	3.34E-02	&	2.00 	&	3.52E+00	&	1.00 	&	8.86E-01	&	1.02 	&	3.98E-01	&	1.02 	\\
				\Xhline{1pt}
				&		&	4	&	1.78E-15	&	8.76E+00	&		&	5.33E+01	&		&	1.30E+03	&		&	6.40E+00	&		\\
				&		&	8	&	6.66E-16	&	2.20E+00	&	2.00 	&	2.79E+01	&	0.93 	&	6.53E+02	&	0.99 	&	3.42E+00	&	0.91 	\\
				0	&	$4$	&	16	&	3.87E-16	&	5.41E-01	&	2.02 	&	1.41E+01	&	0.99 	&	3.27E+02	&	1.00 	&	1.58E+00	&	1.11 	\\
				&		&	32	&	2.91E-16	&	1.34E-01	&	2.01 	&	7.05E+00	&	1.00 	&	1.64E+02	&	1.00 	&	7.53E-01	&	1.07 	\\
				&		&	64	&	1.87E-16	&	3.34E-02	&	2.00 	&	3.52E+00	&	1.00 	&	8.18E+01	&	1.00 	&	3.96E-01	&	0.93 	\\
				\Xhline{1pt}
				&		&	4	&	1.78E-15	&	8.76E+00	&		&	5.33E+01	&		&	1.30E+05	&		&	6.49E+00	&		\\
				0	&		&	8	&	6.66E-16	&	2.20E+00	&	2.00 	&	2.79E+01	&	0.93 	&	6.53E+04	&	0.99 	&	3.42E+00	&	0.93 	\\
				&	$6$	&	16	&	4.44E-16	&	5.41E-01	&	2.02 	&	1.41E+01	&	0.99 	&	3.27E+04	&	1.00 	&	1.66E+00	&	1.04 	\\
				&		&	32	&	3.19E-16	&	1.34E-01	&	2.01 	&	7.05E+00	&	1.00 	&	1.64E+04	&	1.00 	&	7.90E-01	&	1.07 	\\
				&		&	64	&	1.87E-16	&	3.34E-02	&	2.00 	&	3.52E+00	&	1.00 	&	8.18E+03	&	1.00 	&	3.98E-01	&	0.99 	\\
				
				\Xhline{1pt}
				
				&		&	4	&	1.88E-15	&	1.18E+00	&		&	1.40E+01	&		&	5.53E+00	&		&	1.61E+00	&		\\
				&		&	8	&	1.79E-15	&	1.52E-01	&	2.96 	&	3.78E+00	&	1.89 	&	1.19E+00	&	2.22 	&	4.37E-01	&	1.89 	\\
				1	&	$2$	&	16	&	1.05E-15	&	1.94E-02	&	2.97 	&	1.03E+00	&	1.88 	&	2.74E-01	&	2.12 	&	1.11E-01	&	1.98 	\\
				&		&	32	&	8.90E-16	&	2.45E-03	&	2.98 	&	2.91E-01	&	1.82 	&	6.89E-02	&	1.99 	&	2.77E-02	&	2.00 	\\
				&		&	64	&	4.72E-16	&	3.12E-04	&	2.98 	&	8.73E-02	&	1.74 	&	1.86E-02	&	1.89 	&	6.98E-03	&	1.99 	\\
				\Xhline{1pt}
				&		&	4	&	1.72E-15	&	1.18E+00	&		&	1.40E+01	&		&	1.25E+02	&		&	1.65E+00	&		\\
				&		&	8	&	1.78E-15	&	1.52E-01	&	2.96 	&	3.78E+00	&	1.89 	&	3.14E+01	&	1.99 	&	4.37E-01	&	1.92 	\\
				1	&	$4$	&	16	&	1.07E-15	&	1.94E-02	&	2.97 	&	1.03E+00	&	1.88 	&	7.87E+00	&	2.00 	&	1.11E-01	&	1.98 	\\
				&		&	32	&	8.90E-16	&	2.45E-03	&	2.98 	&	2.91E-01	&	1.82 	&	1.97E+00	&	2.00 	&	2.79E-02	&	1.99 	\\
				&		&	64	&	4.55E-16	&	3.12E-04	&	2.98 	&	8.73E-02	&	1.74 	&	4.92E-01	&	2.00 	&	6.98E-03	&	2.00 	\\
				\Xhline{1pt}
				&		&	4	&	1.65E-15	&	1.18E+00	&		&	1.40E+01	&		&	1.25E+04	&		&	1.65E+00	&		\\
				1	&		&	8	&	1.78E-15	&	1.52E-01	&	2.96 	&	3.78E+00	&	1.89 	&	3.14E+03	&	1.99 	&	4.37E-01	&	1.92 	\\
				&	$6$	&	16	&	1.03E-15	&	1.94E-02	&	2.97 	&	1.03E+00	&	1.88 	&	7.87E+02	&	2.00 	&	1.11E-01	&	1.98 	\\
				&		&	32	&	8.92E-16	&	2.45E-03	&	2.98 	&	2.91E-01	&	1.82 	&	1.97E+02	&	2.00 	&	2.79E-02	&	1.99 	\\
				&		&	64	&	4.58E-16	&	3.12E-04	&	2.98 	&	8.73E-02	&	1.74 	&	4.92E+01	&	2.00 	&	6.98E-03	&	2.00 	\\

				\Xhline{1pt}
			\end{tabular}
		\end{table}

		\begin{table}[H]\small
	\caption{Example \ref{example1}:  Pressure-robustness: Errors and observed convergence orders for the dual pressure $q$,  dual state $\bm z$,  and its flux $\mathbb G$ by using the new HDG formulation \eqref{HDG_discrete2}.}
			\label{Table2}
			\centering
			\begin{tabular}{c|c|c|c|c|c|c|c|c|c}
				\Xhline{1pt}
				
				\multirow{2}{*}{$k$} &
				\multirow{2}{*}{$n$} &
				\multirow{2}{*}{$\frac{\sqrt{2}}{h}$} &
				\multirow{2}{*}{$\textup{div}(\bm z_h)$}&
				
				\multicolumn{2}{c|}{$\|\bm z-\bm z_h\|_{L^2(\Omega)}$} &

				\multicolumn{2}{c|}{$\|\mathbb G-\mathbb G_h\|_{L^2(\Omega)}$}&
				
				\multicolumn{2}{c}{$\|q-q_h\|_{L^2(\Omega)}$}

				\\
				
				\cline{5-10}
				& & 	& &Error &Rate  &Error &Rate &Error &Rate    \\

				\Xhline{1pt}
				
				&		&	4	&	1.11E-16	&	8.51E-01	&		&	6.27E+00	&		&	1.31E+01	&		\\
				&		&	8	&	5.55E-17	&	2.47E-01	&	1.78 	&	3.37E+00	&	0.90 	&	6.59E+00	&	0.99 	\\
				0	&	$2$	&	16	&	3.23E-17	&	6.55E-02	&	1.92 	&	1.71E+00	&	0.97 	&	3.29E+00	&	1.00 	\\
				&		&	32	&	2.31E-17	&	1.68E-02	&	1.96 	&	8.60E-01	&	1.00 	&	1.64E+00	&	1.00 	\\
				&		&	64	&	1.61E-17	&	4.24E-03	&	1.98 	&	4.30E-01	&	1.00 	&	8.20E-01	&	1.00 	\\
				\Xhline{1pt}
				
				&		&	4	&	1.11E-16	&	8.51E-01	&		&	6.27E+00	&		&	1.30E+03	&		\\
				&		&	8	&	5.55E-17	&	2.47E-01	&	1.78 	&	3.37E+00	&	0.90 	&	6.53E+02	&	0.99 	\\
				0	&	$4$	&	16	&	4.36E-17	&	6.55E-02	&	1.92 	&	1.71E+00	&	0.97 	&	3.27E+02	&	1.00 	\\
				&		&	32	&	2.29E-17	&	1.68E-02	&	1.96 	&	8.60E-01	&	1.00 	&	1.64E+02	&	1.00 	\\
				&		&	64	&	1.39E-17	&	4.24E-03	&	1.98 	&	4.30E-01	&	1.00 	&	8.18E+01	&	1.00 	\\
				\Xhline{1pt}
				
				&		&	4	&	1.11E-16	&	8.51E-01	&		&	6.27E+00	&		&	1.30E+05	&		\\
				0	&		&	8	&	5.55E-17	&	2.47E-01	&	1.78 	&	3.37E+00	&	0.90 	&	6.53E+04	&	0.99 	\\
				&	$6$	&	16	&	3.72E-17	&	6.55E-02	&	1.92 	&	1.71E+00	&	0.97 	&	3.27E+04	&	1.00 	\\
				&		&	32	&	2.08E-17	&	1.68E-02	&	1.96 	&	8.60E-01	&	1.00 	&	1.64E+04	&	1.00 	\\
				&		&	64	&	1.39E-17	&	4.24E-03	&	1.98 	&	4.30E-01	&	1.00 	&	8.18E+03	&	1.00 	\\
				\Xhline{1pt}
				
				&		&	4	&	1.59E-16	&	1.62E-01	&		&	1.93E+00	&		&	1.49E+00	&		\\
				&		&	8	&	1.34E-16	&	2.12E-02	&	2.93 	&	5.06E-01	&	1.93 	&	3.60E-01	&	2.05 	\\
				1	&	$2$	&	16	&	9.26E-17	&	2.70E-03	&	2.97 	&	1.28E-01	&	1.99 	&	8.76E-02	&	2.04 	\\
				&		&	32	&	6.39E-17	&	3.41E-04	&	2.99 	&	3.20E-02	&	2.00 	&	2.17E-02	&	2.02 	\\
				&		&	64	&	3.92E-17	&	4.27E-05	&	3.00 	&	8.01E-03	&	2.00 	&	5.39E-03	&	2.01 	\\
				\Xhline{1pt}
				&		&	4	&	1.64E-16	&	1.62E-01	&		&	1.93E+00	&		&	1.25E+02	&		\\
				&		&	8	&	1.30E-16	&	2.12E-02	&	2.93 	&	5.06E-01	&	1.93 	&	3.14E+01	&	1.99 	\\
				1	&	$4$	&	16	&	8.68E-17	&	2.70E-03	&	2.97 	&	1.28E-01	&	1.99 	&	7.87E+00	&	2.00 	\\
				&		&	32	&	6.72E-17	&	3.41E-04	&	2.99 	&	3.20E-02	&	2.00 	&	1.97E+00	&	2.00 	\\
				&		&	64	&	3.84E-17	&	4.27E-05	&	3.00 	&	8.01E-03	&	2.00 	&	4.92E-01	&	2.00 	\\
				\Xhline{1pt}
				&		&	4	&	1.64E-16	&	1.62E-01	&		&	1.93E+00	&		&	1.25E+04	&		\\
				1	&		&	8	&	1.26E-16	&	2.12E-02	&	2.93 	&	5.06E-01	&	1.93 	&	3.14E+03	&	1.99 	\\
				&	$6$	&	16	&	8.52E-17	&	2.70E-03	&	2.97 	&	1.28E-01	&	1.99 	&	7.87E+02	&	2.00 	\\
				&		&	32	&	6.37E-17	&	3.41E-04	&	2.99 	&	3.20E-02	&	2.00 	&	1.97E+02	&	2.00 	\\
				&		&	64	&	3.95E-17	&	4.27E-05	&	3.00 	&	8.01E-03	&	2.00 	&	4.92E+01	&	2.00 	\\

				\Xhline{1pt}
			\end{tabular}
		\end{table}

	}

	\begin{example}\label{example2}
		Next, we test the problem  {with unknown true solutions}. We use the same data from \cite[Example 5.1]{GongStokes_Tangential1}. We set $\Omega=(0,0.125)^2$, $\color{black}\bm f= \bm 0$, and $\gamma=1$. To  {show that} our HDG method is pressure-robust, we   {perturb} the target state $\bm y_d$ by a large gradient field $\phi = 10^6(x+y)$. We take
\begin{align*}
\bm y_d&=200\times 8^3[x^2(1-8x)^2\textcolor{black}{y}(1-8y)(1-16y), -x(1-8x)(1-16x)y^2(1-y)^2]^\top, \\
\widetilde{\bm y_d}&=\bm y_d+ 10^6 [1,1]^\top.
\end{align*}
We denote the corresponding velocity by $\bm y$ and $\widetilde{\bm y}$.  We know the fact that perturbing the external force by a gradient field affects only the pressure, and not the velocity; this was shown in \cite{MR3133522}. Hence, $\bm y = \widetilde{\bm y}$.

We first   {solve} the optimality system  \eqref{optimality_system} by using the HDG method proposed in \cite{GongStokes_Tangential1} with $h = \frac{\sqrt 2}{1024}$and $k=1$ for both $\bm y_d$ and $\widetilde{\bm y_d}$, we compute the difference of $\bm y_h$ and $\widetilde{\bm y_h}$:
\begin{align*}
\|\bm y_h - \widetilde{\bm y_h}\|_{L^2(\Omega)} =  214.
\end{align*}
Next, we use the HDG formulation \eqref{HDG_discrete2} in this paper, and we have
\begin{align*}
\|\bm y_h - \widetilde{\bm y_h}\|_{L^2(\Omega)} =  6.94\times 10^{-7}.
\end{align*}

We see that the algorithm proposed and analyzed in \cite{GongStokes_Tangential1} is not pressure-robust; while the algorithm \eqref{HDG_discrete2} is pressure-robust.

\end{example}
	\section{Conclusion}
{	
	In \cite{GongStokes_Tangential1}, we used an existing HDG method to approximate the solution of a tangential Dirichlet boundary control problem for the Stokes system. The velocities were not in  $\bm H(\textup{div}; \Omega)$ and the error estimates depended on the pressures. 	In this work, we  devised a new globally divergence free and pressure-robust HDG method for solving this problem. We proved that the discrete velocity belongs to $\bm H(\textup{div}; \Omega)$ and is globally  divergence free. Furthermore, our error estimates show that the errors for the control and velocities do not depend on the pressures.
	
	As far as we are aware, this is the first work to obtain a global divergence free and pressure-robust numerical method for an optimal boundary control problem involving Stokes equations. In the future, we will consider  devising pressure-robust numerical methods when  using an energy space for the control \cite{GongStokes_energy1}. Besides that, we plan to devise divergence free and pressure-robust HDG schemes for  more complicated PDEs, such as the Oseen and Navier-Stokes equations; and apply the methods to other PDE optimal control problems.}

\end{document}